\documentclass[10pt, reqno]{amsart}
\usepackage[utf8]{inputenc}

\usepackage{style}
\title{Sharp bounds for the fractional chromatic number of high-girth $d$-degenerate graphs}

\author{Peter Allen$^\star$}
\address{$^\star$Department of Mathematics, The London School of Economics and Political Science}
\email{p.d.allen@lse.ac.uk}

\author{Abhishek Dhawan$^\dagger$}
\address{$^\dagger$Department of Mathematics, University of Illinois Urbana--Champaign}
\email{adhawan2@illinois.edu}
\thanks{The research of the second author is supported by NSF RTG grant DMS-1937241.}

\author{Jonathan A. Noel$^\ddagger$}
\address{$^\ddagger$Department of Mathematics and Statistics, The University of Victoria}
\thanks{The research of the third author is supported by NSERC Discovery Grant RGPIN-2021-02460.}
\email{noelj@uvic.ca}

\date{}

\begin{document}

\begin{abstract}
    Martinsson and Steiner recently proved that the fractional chromatic number of any $d$-degenerate triangle-free graph $G$ satisfies $\chi_f(G) = O\left(\frac{d}{\log d}\right)$. They further conjectured a sharp leading constant $1 + o(1)$. In this paper, we confirm their upper bound conjecture for graphs having girth at least $5$. Our proof is constructive: it gives an efficient randomized algorithm that, with high probability, computes a fractional coloring of weight at most $(1 + o(1))\frac{d}{\log d}$ in such graphs.
    
    Furthermore, we establish their conjectured lower bound in a stronger form: for any constant $g \ge 4$, there exist $d$-degenerate graphs having girth at least $g$ with $\chi_f(G) \ge (1 - o(1))\frac{d}{\log d}$. This lower bound is achieved by analyzing a random graph based on the uniform attachment model. Notably, our results reveal that this model lacks the typical computational complexity barriers found in Erd\H{o}s-R\'enyi graphs, where there is a conjectured factor-$2$ algorithmic gap for this problem. 
\end{abstract}

\maketitle

\section{Introduction}

\subsection{Background}

Understanding the chromatic number, denoted by $\chi(G)$, of a graph $G$ is a classical and well-studied topic in graph theory, beginning with the standard inequalities
\[\omega(G)\le\chi(G)\le \deg(G)+1\le\Delta(G)+1,\]
where $\omega(G)$ is the number of vertices in a largest clique of $G$, $\Delta(G)$ is the maximum degree of $G$, and $\deg(G)$ is the \emph{degeneracy} of $G$, defined by $\deg(G)\coloneqq \max_{G'\subseteq G}\delta(G)$. These inequalities can all be equal:
this occurs if and only if a component of $G$ is a clique on $\Delta(G)+1$ vertices.

It is then natural to ask what happens if $\omega(G)$ and $\deg(G)$, or $\Delta(G)$, are far apart: the first case to consider is $\omega(G)=2$, i.e., $G$ is triangle-free. A seminal result of Johansson~\cite{johansson} is that in this case we have $\chi(G)=O\big(\frac{\Delta}{\log\Delta}\big)$, saving a $\log$ factor from the trivial bound.\footnote{While commonly cited as a DIMACS Technical Report, no record of Johansson's manuscript exists. Instead, Molloy and Reed~\cite{MolloyReed} present a variant of Johansson's proof demonstrating that $\chi(G) \leq 160\frac{\Delta}{\log \Delta}$ for triangle-free $G$.} Here and throughout the paper, $\log$ denotes the natural logarithm. The implied constant was improved to $1+o(1)$ by Molloy~\cite{Molloy}, and analysis of random regular graphs shows that $\tfrac12+o(1)$, if true, would be best possible. Despite considerable work~\cite{AndersonBernshteynDhawan, Kttt, campos2023new, bernshteyn2019johansson, DKPS, bernshteyn2023counting, hurley2023uniformly} this factor-$2$ gap remains the state of the art.

Molloy's approach is to give a polynomial-time algorithm which finds a coloring of $G$ with $(1+o(1))\tfrac{\Delta}{\log\Delta}$ colors. In particular, such an algorithm necessarily finds an independent set in $G$ of size at least $(1-o(1))\frac{n\log\Delta}{\Delta}$, matching asymptotically a foundational result of Shearer~\cite{shearer}\footnote{Shearer proves this result in a stronger form with $\Delta$ replaced by the average degree of $G$.}. We do not know any construction of triangle-free graphs with independence number smaller than $(2-o(1))\frac{n\log\Delta}{\Delta}$, and a conjecture in computational complexity asserts that no randomized polynomial-time algorithm can find independent sets in all triangle-free graphs of size larger than $(1+\eps)\frac{n\log\Delta}{\Delta}$ for any $\eps > 0$~\cite{Karp}. As all current arguments yield efficient algorithms, any improvement on Molloy's result would require a very different approach assuming this conjecture.

It is worth pointing out that we know two quite different constructions which do achieve independence number $(2-o(1))\frac{n\log\Delta}{\Delta}$. This is the independence number of the random $\Delta$-regular graph~\cite{FL}, and if $\Delta\ll\sqrt{n}$ it is not hard to show that we can delete a few edges to remove all triangles without affecting the independence number much; it is also the independence number of the triangle-free process (which is quite similar)~\cite{FGM,BohmanKeevash}. This construction fails badly when $\Delta$ is larger: recently Hefty, Horn, King, and Pfender~\cite{HHKP} gave a triangle-free construction that attains independence number $(2-o(1))\frac{n\log\Delta}{\Delta}$ for slightly larger $\Delta$, which in particular achieves the current best lower bound on the Ramsey number $R(3,k)$. 

A natural question to ask is then whether we can improve bounds on $\chi(G)$ by considering the degeneracy instead of the maximum degree. Unfortunately, as observed by Alon, Krivelevich, and Sudakov~\cite{AKS}, the answer is no: a variant of Tutte's construction gives triangle-free graphs of degeneracy $d$ whose chromatic number is the maximum $d+1$. It is interesting to note, though, that the graphs they construct have superexponential in $d$ many vertices, and Brada\v{c}, Fox, Steiner, Sudakov, and Zhang~\cite{BFSSZ} showed that this is necessary: if we add the condition $n\le e^{d^{1-\eps}}$ for any $\eps>0$, a generalization of Johansson's result holds in terms of degeneracy, i.e., $\chi(G) = O\big(\frac{d}{\log d }\big)$.

Since the chromatic number can be large for degenerate triangle-free graphs, it is natural to ask whether the same still holds for the LP-relaxation, the fractional chromatic number $\chi_f(G)$. This is defined to be the solution to the following linear program:
\begin{mini}
    {x \in \R^{|\mathcal{I}(G)|}}{\sum_{I \in \mathcal{I}(G)}x_I}{\label{eq:chif}}{}
    \addConstraint{\sum_{v\in I \in \mathcal{I}(G)}x_I}{\geq 1}{\qquad v \in V(G)}
    \addConstraint{x_I}{\geq 0}{\qquad I \in \mathcal{I}(G)}
\end{mini}
(Here, $\mathcal{I}(G)$ denotes the collection of all independent sets of $G$.)
As usual, observe that a proper coloring of $G$ gives an integer feasible solution to this linear program, and hence $\chi_f(G)\le\chi(G)$.
Furthermore, we note that a feasible solution $x$ to \eqref{eq:chif} with value $k$ yields a probability distribution $\mathcal{D}$ over $\mathcal{I}(G)$ satisfying $\Pr_{I \sim \mathcal{D}}[v\in I] \geq 1/k$; namely, select independent set $I$ with probability $x_I/k$.

An obvious lower bound on the fractional chromatic number (and hence the chromatic number) is $\chi(G)\ge \chi_f(G)\ge\tfrac{n}{\alpha(G)}$, where $\alpha(G)$ is the independence number of $G$. For the best construction known of bounded-degree triangle-free graphs with high chromatic number, coming from random regular graphs, these inequalities are tight, and Molloy's bound remains asymptotically the best we can do in terms of bounding the fractional chromatic number of triangle-free graphs with given maximum degree. In other words, moving to the fractional chromatic number does not help us when we study triangle-free graphs with given maximum degree. Furthermore, we expect there to be a computational complexity barrier to improving the upper bound in this setting: an efficient randomized algorithm should not be able to improve.

Nevertheless, it is not too hard to check that the Alon-Krivelevich-Sudakov construction does not have large fractional chromatic number, and Harris~\cite{harris} conjectured that $\chi_f(G)=O\big(\frac{d}{\log d}\big)$ for $d$-degenerate triangle-free graphs. This was proved by Martinsson and Steiner~\cite{MartStein}, who showed
\begin{align}\label{eq: MartStein}
    \chi_f(G)\le(4+o(1))\frac{d}{\log d}
\end{align}
holds for all $d$-degenerate triangle-free graphs. They also observed that, since a $d$-regular graph is $d$-degenerate, we cannot hope to improve the $4+o(1)$ below $\tfrac12+o(1)$. They posed the following conjecture for the correct answer.

\begin{conjecture}[{\cite[Conjecture~5.2]{MartStein}}]\label{conj:MS}
The following holds for any sufficiently large $d$.
\begin{enumerate}[label=\ep{\roman*}]
    \item\label{conj: ub} $\chi_f(G)\le(1+o(1))\frac{d}{\log d}$ for all $d$-degenerate triangle-free graphs $G$.
    \item\label{conj: lb} There exists a triangle-free $d$-degenerate graph $G$ with $\chi_f(G)\ge(1-o(1))\frac{d}{\log d}$.
\end{enumerate}
\end{conjecture}

In this paper, we resolve Conjecture~\ref{conj:MS}\ref{conj: lb}. In fact, we construct a $d$-degenerate graph $G$ having girth at least $g$ for constant $g \geq 4$ with $\chi_f(G)\ge(1-o(1))\frac{d}{\log d}$. Additionally, we establish Conjecture~\ref{conj:MS}\ref{conj: ub} under the additional restriction that $G$ has girth at least $5$, i.e., that there are no triangles or $4$-cycles in $G$. Notably, our upper bound proof yields an efficient randomized algorithm that returns a fractional coloring.
In light of our construction, there is no computational complexity barrier in our random model.
This is not entirely unsurprising; we discuss this further in the concluding remarks, Section~\ref{section: concluding remarks}.

\begin{theorem}[Informal versions of Theorems~\ref{theo: girth 5} and \ref{theo: rg}]\label{theo: informal}
    The following hold for any $g \geq 4$ and sufficiently large $d$.
    \begin{enumerate}[label=\ep{\roman*}]
        \item\label{informal ub} $\chi_f(G)\le(1+o(1))\frac{d}{\log d}$ for all $d$-degenerate graphs $G$ having girth at least $5$.
        \item\label{informal lb} There exists a $d$-degenerate graph $G$ having girth at least $g$ with $\chi_f(G)\ge(1-o(1))\frac{d}{\log d}$.
    \end{enumerate}
\end{theorem}

It is reasonable to ask what exactly is meant by an efficient algorithm which returns a vector indexed by exponentially many independent sets. In this paper, the answer is as follows. For $\alpha \in [0, 1]$ and $q \in \N$, we define an $(\alpha,q)$-coloring of $G$ to be a collection of $q$ independent sets in $G$, such that each vertex is in at least $\alpha q$ of the independent sets. Observe that given an $(\alpha,q)$-coloring of $G$, assigning the weight $1/(\alpha q)$ to each of the given $q$ independent sets (and weight $0$ on all other independent sets of $G$) is a feasible solution to~\eqref{eq:chif} with value $1/\alpha$, and hence a witness $\chi_f(G)\le 1/\alpha$. Our algorithm will, with high probability, compute an $(\alpha,q)$-coloring of $G$ in $\poly(n,d)$ time where $q$ is roughly $\frac{d\log n}{\log d}$, which in particular is a $\poly(n)$-sized witness of the fractional chromatic number. Of course, it is easy to sample efficiently from this distribution.

\subsection{Main results}

In this section, we provide formal statements of our main results.
We begin with the upper bound, i.e., a formal version of Theorem~\ref{theo: informal}\ref{informal ub}.
In fact, we prove the bound for $C_4$-free graphs, a weaker condition than having girth at least $5$.

\begin{theorem}\label{theo: girth 5}
    For all $0<\eps < 1$ there exists $d_0 \in \N$ such that the following holds for $d \geq d_0$ and $n \in \N$.
    Let $G$ be an $n$-vertex $d$-degenerate $C_4$-free graph.
    Then, \[\chi_f(G) \leq (1+\eps)\dfrac{d}{\log d}.\]
    Moreover, there is a $\poly(n,d)$-time randomized algorithm that, with high probability, outputs a $\left((1-\eps)\frac{\log d}{d},\,q\right)$-coloring of $G$ with $q=\Theta_\eps\left(d\frac{\log n}{\log d}\right)$.
\end{theorem}

We note that by a result of Esperet, Kang, and Thomass\'e~\cite[Theorem~3.1]{EKT}, a corollary of Theorem~\ref{theo: girth 5} is that any $C_4$-free graph $G$ with minimum degree $d$ contains an induced bipartite subgraph with average degree at least $(1-o(1))\log d$. To see that this is true, observe that we can assume $G$ is minimal subject to having minimum degree at least $d$, since otherwise we can pass to a proper induced subgraph with the same conditions. Such a graph is $d$-degenerate and hence our result applies. It would be interesting to know whether average degree $(1-o(1))\log d$ is optimal.

The corresponding lower bound statement in fact allows us to construct degenerate graphs with arbitrarily high girth and large fractional chromatic number. Recall that the \emph{Lambert $W$-function} is defined by $W(d)e^{W(d)}=d$ for $d>0$ and the \emph{$n$-th harmonic number} for $n\geq1$ is $H_n\coloneqq \sum_{j=1}^n\frac{1}{j}$. It is well known that $W(d)=\log d - \log \log d + o(1)$ and $H_n=\log n+O(1)$ for large $d$ and $n$.

\begin{theorem}\label{theo: rg}
For any $g\ge4$ and $0<\delta< 1$, suppose that $n$ and $d$ satisfy $n\ge d\ge 2$ and
\begin{equation}\label{eq:rg}\frac{1}{g-1}\binom{2g-4}{g-2}d^{g-1}\le\delta\frac{n}{H_n^2}\,.\end{equation}
Then there exists a $d$-degenerate graph $G$ with at most $n$ vertices and girth at least $g$ such that 
\[\chi_f(G)\geq (1-\delta)\frac{d}{W(d)+1}\,.\]
\end{theorem}

We note that the graphs we construct have independent sets larger than $\frac{n}{\chi_f(G)}$.
In particular, the bound on $\chi_f(G)$ does not follow immediately from an upper bound on $\alpha(G)$.
The condition on the degeneracy is essential in exhibiting this gap.
We discuss this further in the following subsection, where we provide an informal overview of our proof techniques.

\subsection{Proof overview}\label{subsection: proof overview}

To prove the upper bound (Theorem~\ref{theo: girth 5}), we use an algorithm developed by the second author~\cite{Dhawan} to give an alternative proof (and several extensions) of Martinsson and Steiner's result~\cite{MartStein}. As we discuss in the concluding remarks, however, this algorithm cannot improve the $4+o(1)$ value in~\eqref{eq: MartStein} without the additional assumption of $C_4$-freeness. The novelty in this paper is to show how to use $C_4$-freeness in the analysis to permit optimal constant choices in the algorithm and hence the sharp upper bound.

For the lower bound, we use a random construction based on the \emph{uniform attachment model} studied in, e.g.,~\cite{FriezePerezGimenezPralatReiniger19,MalyshkinZhukovskii22,AcanPittel20}. The rough idea is to let the vertex set be $[n]=\{1,\dots,n\}$, and for each vertex $i$ choose a random set of $d$ vertices in $[i-1]$ to be adjacent to $i$. It is not hard to see that (if $d$ is not too large with respect to $n$) this random construction is likely to contain very few short cycles and we can remove them by deleting $o(n)$ many vertices.

Interestingly, to provide a lower bound on the fractional chromatic number of this random construction $G$, it is not enough to use the inequality $\chi_f(G)\ge\frac{n}{\alpha(G)}$ as with regular graphs. The random construction contains independent sets on considerably more than $\frac{n\log d}{d}$ vertices, but these independent sets all contain many vertices in $\{n/2,\dots,n\}$ and few in $\{1,\dots,n/2\}$, so that they do not fractionally pack to make a fractional coloring. See the concluding remarks, Section~\ref{section: concluding remarks}, for further discussion.

Instead, we make use of LP-duality. The dual program to the fractional chromatic number is the \emph{fractional clique number} $\omega_f(H)$, which is defined to be the solution to the following linear program.

\begin{maxi}
    {f \in \R^{|V(H)|}}{\sum_{v\in V(H)}f_v}{\label{eq:fracclique}}{}
    \addConstraint{\sum_{v\in I}f_v}{\leq 1}{\qquad I \in \mathcal{I}(H)}
    \addConstraint{f_v}{\geq 0}{\qquad v \in V(H)}
\end{maxi}

As with the fractional chromatic number, an integer feasible assignment for this linear program is necessarily a clique in $H$, so $\omega(H)\le\omega_f(H)$. The strong duality theorem for linear programming states that $\omega_f(H)=\chi_f(H)$ for all graphs $H$, though what we will actually need is only the weak duality statement $\omega_f(H)\le\chi_f(H)$. We will give an explicit (and not random) assignment of values $f_v$ and argue that with high probability this assignment is a fractional clique in the random construction $G$.

\subsection{Notation and terminology}

Throughout the rest of the paper we use the following basic notation. For $n \in \N$, we let $[n] \defeq \set{1, \ldots, n}$.
For a graph $G$, its vertex and edge sets are denoted $V(G)$ and $E(G)$, respectively.
For a vertex $v \in V(G)$, $N_G(v) \coloneqq \set{u \in V(G)\,:\, uv \in E(G)}$ denotes the set of neighbors of $v$, and $d_G(v) \coloneqq |N_G(v)|$ denotes the degree of $v$; we drop the subscript $G$ when the context is clear.
For a subset $U \subseteq V(G)$, the subgraph induced by $U$ is denoted by $G[U]$.

We say an $n$-vertex graph $G$ is \emphd{$d$-degenerate} for $d\in \N$ if there exists an ordering $(v_1, \ldots, v_n)$ of $V(G)$ such that for each $i\in [n]$, we have $d_{G_i}(v_i) \leq d$, where $G_i \coloneqq G[\set{v_1, \ldots, v_i}]$.
Given a degeneracy ordering of $G$, we let $N_L(v_i)$ be the set of neighbors of $v_i$ in the graph $G_i$; we call such vertices \emph{left-neighbors} of $v_i$.
Similarly, we let $N_R(v_i)$ be the vertices $v_k$ such that $v_i \in N_L(v_k)$, i.e., the \emph{right-neighbors} of $v_i$. For a set $X$ of vertices, define $N_L(X):=\bigcup_{v\in X}N_L(v)$ and $N_R(X):=\bigcup_{v\in X}N_R(v)$.

\subsubsection*{Structure of the paper}
The rest of the paper is organized as follows: in Section~\ref{section: alg}, we prove Theorem~\ref{theo: girth 5}; in Section~\ref{section: rg}, we prove Theorem~\ref{theo: rg}; finally, in Section~\ref{section: concluding remarks}, we conclude with a discussion of our approach and outline potential future avenues of work.

\section{$C_4$-free graphs: proof of Theorem~\ref{theo: girth 5}}\label{section: alg}

In this section, we will prove Theorem~\ref{theo: girth 5}, which bounds the fractional chromatic number of $d$-degenerate $C_4$-free graphs.
For the reader's convenience, we restate the result below.

\begin{theorem*}[Restatement of Theorem~\ref{theo: girth 5}]
    For all $0<\eps < 1$ there exists $d_0 \in \N$ such that the following holds for $d \geq d_0$ and $n \in \N$.
    Let $G$ be an $n$-vertex $d$-degenerate $C_4$-free graph.
    Then, \[\chi_f(G) \leq (1+\eps)\dfrac{d}{\log d}.\]
    Moreover, there is a $\poly(n,d)$-time randomized algorithm that, with high probability, outputs a $\left((1-\eps)\frac{\log d}{d},\,q\right)$-coloring of $G$ with $q=\Theta_\eps\left(d\frac{\log n}{\log d}\right)$.
\end{theorem*}

Let us describe our coloring procedure, which relies on a key subroutine that samples an independent set in $G$ with specific properties (see Algorithm~\ref{algorithm: fcp}).
The algorithm takes as input an $n$-vertex $d$-degenerate $C_4$-free graph $G$ and a parameter $0<\eps < 1$, and outputs a random independent set $I$ in $G$. The key property of the algorithm is that for each $v\in V(G)$ we have $\Pr[v\in I]$ is very close to $\frac{\log d}{d}$. It follows that $q$ independent runs of the algorithm, if $q$ is sufficiently large, with high probability gives us a collection $\{I_c\,:\,c\in[q]\}$ of independent sets in $G$ which cover each $v\in V(G)$ at least $(1-\eps)\frac{\log d}{d}q$ times, and hence easily a $\left((1-\eps)\frac{\log d}{d},q\right)$-coloring of $G$. 

The basic idea of the algorithm is as follows. We initialize at time $0$ each vertex of $G$ with a weight $p_0(v)=\alpha$ slightly smaller than $\frac{\log d}{d}$. We process the vertices of $G$ in degeneracy order: when at time $i$ we process vertex $v_i$, with probability $p_{i-1}(v_i)$ we select $v_i$ into $I$. We then update the weights of its right neighbors: if $v_i$ is selected, we set their weights to $0$, and if not we multiply their weights by $(1-p_{i-1}(v_i))^{-1}$. This factor is chosen such that for each $\ell$ and $1<i<\ell$, we have the martingale property $\E[p_i(v_\ell)\mid p_{i-1}(\cdot)]=p_{i-1}(v_\ell)$.

If the algorithm worked as just written, by construction it would output an independent set (once a weight is set to zero, it can never become nonzero) and by the martingale property we would have $\Pr[v\in I]=\alpha$ for all $v$, as we want. However, there is a problem: if $p_{i}(v)$ exceeds $1$ for some $v$, the algorithm fails. In fact, for our analysis we want $p_i(v)$ to stay significantly less than $1$.

To avoid this problem, we set a threshold $\hat{p}$; if $v$ is a right neighbor of $v_i$ and updating $p_{i-1}(v)$ would lead to $p_i(v)>\hat{p}$, we set $p_i(v)=\hat{p}$. Of course, this breaks the martingale property. To avoid this, we use equalizing coin flips: if $v_i$ is chosen into $I$, the first sketch would say we set the weight of $v$ to $0$, but in fact we set it equal to either $0$ or $\hat{p}$ with a carefully chosen probability that recovers the martingale property.

In turn, this means we should not select into $I$ any vertex whose weight is $\hat{p}$, as this risks an edge appearing in $I$. We call such vertices \emph{bad}; we cease to update the weights of bad vertices, and do nothing when it comes to processing them.

For each $v_k$, by the martingale property, it follows that $\Pr[v_k\in I] = \alpha - \hat p\Pr[v_k\text{ bad}]$ (see \eqref{eq: exp size S(vk)} below), and so it suffices to argue that $\Pr[v_k\text{ bad}]$ is small.
See \cite[Ch. 13]{MolloyReed} for a more in-depth discussion of the necessity of thresholding such weights in the context of list coloring $K_3$-free graphs.
Let us now describe the algorithm more formally.

\begin{algorithm}[htb!]
    \caption{Random Independent Set Procedure}\label{algorithm: fcp}
    \SetKwInOut{Input}{Input}
    \SetKwInOut{Output}{Output}
    \SetKwInOut{Init}{Initialize}
    \DontPrintSemicolon
    \Input{An $n$-vertex $d$-degenerate $C_4$-free graph $G$ and a parameter $0<\eps < 1$.}
    \Output{An independent set $I\subset V(G)$.}
    \Init{Set $B_0=\emptyset$.
    For each $v \in V(G)$, set $p_0(v) = \alpha$, where $\alpha \coloneqq \dfrac{\log d}{(1 + \eps/2)d}$.\;
    Let $\hat p \coloneqq d^{-\eps/20}$ and fix a degeneracy ordering $(v_1, \ldots, v_n)$ of $V(G)$.}

    \ForEach{$i = 1, \ldots, n$}{
        \lIf{$v_i\in B_{i-1}$}{set $p_i(v_j)=p_{i-1}(v_j)$ for all $j$.}
        \Else{
            \lForEach{$j$ such that $v_j \notin N_R(v_i)$}{
                set $p_i(v_j) = p_{i-1}(v_j)$.
            }
            \lnl{step: ai} Let $a_{i} \sim \mathrm{BER}(p_{i-1}(v_i))$.\;
            \lIf{ $a_{i} = 1$}{add $v_i$ to $I$.}
            \ForEach{ $v_j \in N_R(v_i)\setminus B_{i-1}$}{
                \lnl{step: ai delete} \uIf{$a_{i} = 1$}{
                    set $p_i(v_j) = \hat p$ with probability $\mu_{i, j}\coloneqq \max\left\{0,\, \dfrac{p_{i-1}(v_j)/\hat p - 1 + p_{i-1}(v_i)}{p_{i-1}(v_i)}\right\}$,\;
                    and $p_i(v_j)=0$ otherwise.\;
                    }
                \lElse{set $p_i(v_j) = \min\left\{\dfrac{p_{i-1}(v_j)}{1 - p_{i-1}(v_i)},\, \hat p\right\}$.}
            }
        }
        \lnl{step: bad update} Set $B_i=\{v_j\,:\,p_i(v_j)=\hat{p}\}$.\;
    }
\end{algorithm}

Note that Algorithm~\ref{algorithm: fcp} can be implemented in $O(nd)$ time.
As a result of Steps~\ref{step: ai delete}~and~\ref{step: bad update}, the set $I$ is an independent set at every step of the algorithm, and in particular at the end.
The probability $\mu_{i, j}$ is well-defined: indeed, $\mu_{i, j} \leq 1$ since $p_{i-1}(v_j) \leq \hat p$. We have $\mu_{i, j} > 0$ if and only if $p_{i-1}(v_j)/(1 - p_{i-1}(v_i)) > \hat p$: this is the equalizing coin flip we described earlier.

We now justify the claim that $p_i(v_k)$ has the martingale property.
    \begin{lemma}\label{lemma: potential}
        For $i = 1, \ldots, k-1$, we have $\E[p_i(v_k) \mid p_{i-1}(\cdot)] = p_{i-1}(v_k)$.
    \end{lemma}
    \begin{proof}
        If $v_i \notin N_L(v_k)$ or $v_i\in B_{i-1}$, the claim is trivial.
        
        Suppose $v_i \in N_L(v_k)$ and $v_i\not\in B_{i-1}$.
        We have two cases to consider.
        \begin{itemize}
            \item \textbf{Case 1:} $\frac{p_{i-1}(v_k)}{1 - p_{i-1}(v_i)} \leq \hat p$.
            We have
            \begin{align*}
                \E[p_i(v_k) \mid p_{i-1}(\cdot)] &= (1 - p_{i-1}(v_i))\frac{p_{i-1}(v_k)}{1 - p_{i-1}(v_i)}  = p_{i-1}(v_k)\,.
            \end{align*}
    
            \item \textbf{Case 2:} $\frac{p_{i-1}(v_k)}{1 - p_{i-1}(v_i)} > \hat p$.
            We have
            \begin{align*}
                \E[p_i(v_k) \mid p_{i-1}(\cdot)] = (1 - p_{i-1}(v_i))\hat p+ p_{i-1}(v_i)\,\hat p\,\mu_{i, k}
                = p_{i-1}(v_k)
            \end{align*}
            by the definition of $\mu_{i, k}$.\qedhere
        \end{itemize}
    \end{proof}

The key property of the algorithm is given by the following lemma.

\begin{lemma}\label{lemma: exp size S(vk)} For each $v\in V(G)$ we have
    $\Pr[v\in I] \geq (1 - \eps/40)\alpha$.
\end{lemma}

We defer the proof of this lemma, and explain how Theorem~\ref{theo: girth 5} follows.
We need the following form of the Chernoff bound, which can be found, for instance, in \cite[Ch.~4]{Mitzenmacher}. 

\begin{theorem}[Chernoff]\label{theo: chernoff}
    Let $X$ be a random variable that is the sum of mutually independent indicator variables, and let $\mu = \E[X]$. Then for any value $\delta > 0$,
    \[\Pr\left[X \leq (1 -\delta) \mu\right] \leq \exp \left( - \frac{ \delta^2 \mu }{2 }\right).\]
\end{theorem}

\begin{proof}[Proof of Theorem~\ref{theo: girth 5}]
Given $0<\eps<1$, we set $d_0$ such that $d^{-\eps/100}\log d\le\eps/40$ for all $d\ge d_0$. Given $d\ge d_0$ and $n\in\N$, we set $q=\left\lceil400^2\cdot 20\eps^{-4}\frac{d\log n}{\log d}\right\rceil$. Now let $G$ be an $n$-vertex graph with vertex set $\{v_1,\dots,v_n\}$ in $d$-degeneracy order.

We independently run Algorithm~\ref{algorithm: fcp} $q$ times with input $G$ and $\eps$, generating independent sets $I_1,\dots,I_q$ in $G$. Fix a vertex $v\in V(G)$. 

Consider $X_v=\sum_{c\in[q]}Y_{c}$, where $Y_c$ is the indicator variable of the event $v\in I_c$. By Lemma~\ref{lemma: exp size S(vk)}, for each $c\in[q]$ we have
\[(1-\eps/20)\alpha \leq \frac{(1-\eps/20)\Pr[Y_c]}{(1 - \eps/40)} \leq (1 - \eps^2/400)\Pr[Y_c]\,.\]
Using this and Theorem~\ref{theo: chernoff}, we have
\begin{align*}
    \Pr\left[X_v \leq (1-\eps/20)q\alpha\right] &\leq \Pr\left[X_v \leq (1 - \eps^2/400)\E[X_v]\right] \\
    &\leq \exp \left( - \frac{ \eps^4 \E[X_v]}{400^2\cdot 2}\right)\le\exp\left(-\frac{\eps^4q\alpha}{400^2\cdot 6}\right)\,.
\end{align*}
Now by a union bound over $v\in V(G)$, we conclude
\[\Pr\left[\exists v\in V(G) \text{ s.t. }X_{v} \leq \frac{q\,\log d}{(1+\eps)\,d}\right] \leq \Pr\left[\exists v \text{ s.t.\ }X_v \leq (1-\eps/20)q\alpha\right] \leq n\exp\left(-\frac{\eps^4q\alpha}{400^2\cdot 6}\right).\]
The above is at most $n^{-1}$ by choice of $q$, completing the proof of Theorem~\ref{theo: girth 5}.
\end{proof}

We now prove Lemma~\ref{lemma: exp size S(vk)}. The idea is the following: for each $k$, we will find an upper bound on $\E[p_{k-1}(v_k)^2]$. Observing that $\E[p_{k-1}(v_k)^2]\ge\hat{p}^2\Pr[v\text{ bad}]$, this gives the upper bound on $\Pr[v\text{ bad}]$ we need.

In turn, our route to the upper bound on $\E[p_{k-1}(v_k)^2]$ is to write down an expression involving the weights of neighbors of $v_k$ which (1) we can evaluate at the start of the algorithm, (2) at time $k-1$ is equal to $p_{k-1}(v_k)^2$, and (3) is a supermartingale (i.e., its expected value at time $i$ is bounded above by its value at time $i-1$).

\begin{proof}[Proof of Lemma~\ref{lemma: exp size S(vk)}]
 Given $v\in V(G)$, suppose $v=v_k$. The equality we use is
\begin{align}\label{eq: exp size S(vk)}
    \Pr[v_k\in I] = \E\left[p_{k-1}(v_k)\mathbbm{1}(v_k\not\in B_{k-1})\right] = \E\left[p_{k-1}(v_k)\right] - \hat{p}\Pr[v_k\in B_{k-1}].
\end{align}
We first compute $\E[p_{k-1}(v_k)]=\alpha$. Applying the martingale property Lemma~\ref{lemma: potential} repeatedly, we have
\[\alpha=\E[p_0(v_k)]=\E\left[\E[p_1(v_k)\mid p_0(\cdot)]\right]=\E[p_1(v_k)]=\dots=\E[p_{k-1}(v_k)]\,.\]

We now bound $\E[p_{k-1}(v_k)^2]$ from above. To simplify equations in what follows, define $\eps'=\frac{\eps}{60}$. The supermartingale random variables we use to do this are
\[S_i\coloneqq p_{i}(v_k)^2\prod_{\substack{i < j \leq k-1 \\ v_j \in N_L(v_k)}}\left(1 + (1+(1+2\eta_{i, j, k})\eps')p_{i}(v_j)\right)\,,\]
where $\eta_{i, j, k} = |N_L(v_j) \cap N_L(v_k) \cap \{v_{i+1}, \ldots, v_{j-1}\}|$. (Note that $\eta_{i, j, k} \in \{0, 1\}$ by $C_4$-freeness.)

Observe that $S_{k-1}=p_{k-1}(v_k)^2$, since the product is empty. The critical claim is that this sequence of random variables is a supermartingale.

    \begin{claim*}
        For $i = 1, \ldots, k-1$, we have $\E[S_i \mid p_{i-1}(\cdot)] \leq S_{i-1}$.
    \end{claim*}
    \begin{claimproof}
        Given $i\in[k-1]$, we have
        \begin{equation}\E[S_i \mid p_{i-1}(\cdot)] =\E\left[p_{i}(v_k)^2\prod_{\substack{i< j \leq k-1 \\ v_j \in N_L(v_k)}}\left(1 + \left(1+(1+2\eta_{i, j, k})\eps'\right)p_{i}(v_j)\right) \,\,\bigg|\,\, p_{i-1}(\cdot)\right]. \label{eq: exp bound}
        \end{equation}
        by definition. This expression is easier to handle than it looks: when we process $v_i$, it turns out that at most three of the terms in the expression change, and we just need to show that that change is a supermartingale change. We now do this by considering cases based on the relationship between $v_i$ and the vertices in $X \cup \{v_k\}$, where $X \coloneqq N_L(v_k) \cap \{v_{i+1}, \ldots, v_{k}\}$.
        Note that as $G$ is $C_4$-free, we have $N_L(u) \cap N_L(v) = \emptyset$ for all pairs $u, v \in X$ and $|N_L(w) \cap N_L(v_k)| \leq 1$ for all $w \in X$.
        Two cases are easy.
        
        If $v_i\notin N_L(X\cup \{v_k\})$, or $v_i\in N_L(X)\cap B_{i-1}$, when we process $v_i$ we get $S_i=S_{i-1}$ since no terms change, and hence $\E[S_i\mid p_{i-1}(\cdot)]=S_{i-1}$ in this case.

        If $v_i\in N_L(X)\setminus (B_{i-1} \cup N_L(v_k))$, then by the $C_4$-freeness assumption, there is a unique $v_j\in X$ such that $v_iv_j\in E(G)$, and the only term in $S_i$ different from $S_{i-1}$ is the factor $1+(1+(1+2\eta_{i, j, k})\eps')p_{i}(v_j)$. By Lemma~\ref{lemma: potential} and since $\eta_{i-1, j, k} = \eta_{i, j, k}$, we have
        \[\E\left[1+(1+(1+2\eta_{i, j, k})\eps')p_{i}(v_j)\mid p_{i-1}(\cdot)\right]=1+(1+(1+2\eta_{i-1, j, k})\eps')p_{i-1}(v_j)\,,\]
        and it follows $\E[S_i\mid p_{i-1}(\cdot)]=S_{i-1}$ in this case.

        The remaining case is $v_i\in N_L(v_k)$, which needs a little more work. 
        We first consider the case that $v_i \notin N_L(X)$.
        There are two changes when we go from $S_{i-1}$ to $S_i$: we update the value of $p_{i-1}(v_k)$, and we remove the factor $1+(1+(1+2\eta_{i-1, i, k})\eps')p_{i-1}(v_i)$. It follows that what we need to prove is $\E[p_i(v_k)^2\mid p_{i-1}(\cdot)]\le p_{i-1}(v_k)^2(1+(1+(1+2\eta_{i-1, i, k})\eps')p_{i-1}(v_i))$ to verify the desired bound $\E[S_i\mid p_{i-1}(\cdot)]\le S_{i-1}$.
        In turn, observe that (deterministically) we have
        \[p_i(v_k)^2\le p_i(v_k)\frac{p_{i-1}(v_k)}{1-p_{i-1}(v_i)}\le p_i(v_k)p_{i-1}(v_k)\left(1+(1+(1+2\eta_{i-1, i, k})\eps')p_{i-1}(v_i)\right)\,,\]
        where the first inequality is strict if $v_i\in B_{i-1}$ or if its right-hand side exceeds $p_i(v_k)\hat{p}$, and otherwise is an equality. The second inequality holds since $0\le p_{i-1}(v_i)\le\hat{p}\le\eps/40$, where the last inequality is by choice of $\hat{p}$ and $d\ge d_0$. Taking expectations of this, we get
        \begin{align*}
            \E[p_i(v_k)^2\mid p_{i-1}(\cdot)] &\le \E\left[p_i(v_k)p_{i-1}(v_k)\left(1+(1+(1+2\eta_{i-1, i, k})\eps')p_{i-1}(v_i)\right)\mid p_{i-1}(\cdot)\right]\\
            &=p_{i-1}(v_k)^2\left(1+(1+(1+2\eta_{i-1, i, k})\eps')p_{i-1}(v_i)\right),
        \end{align*}
        as desired, where the equality is by Lemma~\ref{lemma: potential}.

        Now suppose that $v_i \in N_L(X)$ as well. 
        Then by the $C_4$-freeness assumption, there is a unique $v_j\in X$ such that $v_iv_j\in E(G)$.
        Furthermore, it must be the case that $\eta_{i, j, k} = 0$ as $N_L(v_k) \cap N_L(v_j) = \{v_i\}$.
        As above, it suffices to show that
        \begin{align*}
            &\E[p_i(v_k)^2(1+(1+(1+2\eta_{i, j, k})\eps')p_{i}(v_j))\mid p_{i-1}(\cdot)] \\
            &\qquad \qquad \qquad \leq p_{i-1}(v_k)^2(1+(1+(1+2\eta_{i-1, i, k})\eps')p_{i-1}(v_i))(1+(1+(1+2\eta_{i-1, j, k})\eps')p_{i-1}(v_j)).
        \end{align*}
        As before, we note that (deterministically) we have
        \begin{align*}
            &p_i(v_k)^2(1+(1+(1+2\eta_{i, j, k})\eps')p_{i}(v_j)) \\
            &\qquad \leq p_i(v_k)p_{i-1}(v_k)\left(1+(1+(1+2\eta_{i-1, i, k})\eps')p_{i-1}(v_i)\right)\left(1+(1+(1+2\eta_{i, j, k})\eps')\frac{p_{i-1}(v_j)}{1 - p_{i-1}(v_i)}\right) \\
            &\qquad \leq p_i(v_k)p_{i-1}(v_k)\left(1+(1+(1+2\eta_{i-1, i, k})\eps')p_{i-1}(v_i)\right)\left(1+(1+(1+2\eta_{i-1, j, k})\eps')p_{i-1}(v_j)\right),
        \end{align*}
        where we use the fact that $0\le p_{i-1}(v_i)\le\hat{p}\le\eps/40$ and that $\eta_{i-1, j, k} = 1 + \eta_{i, j, k} = 1$.
        The claim now follows by Lemma~\ref{lemma: potential}.

        This covers all the cases, completing the proof.
    \end{claimproof}

    Applying repeatedly the Claim, we have
    \begin{equation}\label{eq:variance}
     \begin{split}
        \E[p_{k-1}(v_k)^2]&=\E[S_{k-1}]=\E\left[\E[S_{k-1}\mid p_{k-2}(\cdot)]\right]\le\E[S_{k-2}]\le\dots\le\E[S_0]\\
        &=p_{0}(v_k)^2\prod_{\substack{1 \leq j \leq k-1 \\ v_j \in N_L(v_k)}}(1 + (1+(1+2\eta_{0, j, k})\eps')p_{0}(v_j)) \\
        &\leq \alpha^2\exp\left((1+3\eps')d\alpha\right),
     \end{split}
    \end{equation}
    where we use the fact that all weights are initially $\alpha$, $|N_L(v_k)| \leq d$, and $\eta_{0, j, k} \leq 1$ by $C_4$-freeness.
    
    We now bound the second term in~\eqref{eq: exp size S(vk)}, i.e., we show that $\hat{p}\Pr[v_k\in B_{k-1}]$ is small. As $\E[p_{k-1}(v_k)^2]\ge\hat{p}^2\Pr[v\in B_{k-1}]$, using~\eqref{eq:variance} we get
    \[\hat p\Pr[v_k\in B_{k-1}] \leq \frac{\E[p_{k-1}(v_k)^2]}{\hat p} \leq \frac{\alpha^2\exp\left((1+\eps/20)d\alpha\right)}{\hat p}.\]
    By definition of $\alpha$ and $\hat p$, we have
    \[\frac{\alpha\exp\left((1+\eps/20)d\alpha\right)}{\hat p} = \frac{\log d}{(1+\eps/2)d^{1 - \eps/20}}\exp\left(\left(\frac{1 + \eps/20}{1 + \eps/2}\right)\log d\right) \leq \frac{\log d}{d^{\eps/100}} \leq \eps/40,\]
    by choice of $d\ge d_0$.
Substituting our two calculated values into~\eqref{eq: exp size S(vk)}, we have
\[\Pr[v_k\in I] = \E[p_{k-1}(v_k)] - \hat p\Pr[v_k\in B_{k-1}] \geq \alpha-\eps\alpha/40= (1-\eps/40)\alpha,\]
completing the proof.
\end{proof}

\section{Degenerate graphs of high fractional chromatic number: proof of Theorem~\ref{theo: rg}}\label{section: rg}

In this section, we will prove Theorem~\ref{theo: rg}, which exhibits a $d$-degenerate graph having high girth and high fractional chromatic number.
For the reader's convenience, we restate the result below.

\begin{theorem*}[{Restatement of Theorem\ref{theo: rg}}]
For any $g\ge4$ and $0<\delta< 1$, suppose that $n$ and $d$ satisfy $n\ge d\ge 2$ and
\begin{equation}\frac{1}{g-1}\binom{2g-4}{g-2}d^{g-1}\le\delta\frac{n}{H_n^2}\,.\end{equation}
Then there exists a $d$-degenerate graph $G$ with at most $n$ vertices and girth at least $g$ such that 
\[\chi_f(G)\geq (1-\delta)\frac{d}{W(d)+1}\,.\]
\end{theorem*}

The construction of the graph $G$ in the proof of Theorem~\ref{theo: rg} is based on the \emph{uniform attachment model} studied in, e.g.,~\cite{FriezePerezGimenezPralatReiniger19,MalyshkinZhukovskii22,AcanPittel20}. We will consider a slight variant of this model which is more convenient for certain calculations. 

\begin{defn}
\label{def: uniform attachment}
Let $U_{n,d}$ be the random graph defined as follows. Start with the vertex set $\{0,\dots,n\}$ and, for each $1\leq i\leq n$, let $N_i^-$ be a uniformly random subset of $\{0,\dots,i-1\}$ of cardinality $\min\{i,d\}$ chosen independently of all previous random choices and add an edge from $i$ to every vertex of $N_i^-$. Finally, delete vertex $0$ so that the vertex set of $U_{n,d}$ is precisely $[n]\coloneqq \{1,\dots,n\}$. 
\end{defn}

\begin{remark}
It is clear that $U_{n,d}$ is $d$-degenerate with probability $1$. Indeed, $1,2,\dots,n$ is a valid degeneracy ordering. 
\end{remark}

As discussed in Section~\ref{subsection: proof overview}, we will take advantage of the weak duality theorem for linear programming to prove the desired lower bound.
In particular, recalling the linear programming formulation \eqref{eq:chif} for fractional coloring, it suffices to exhibit an appropriate feasible point to the dual linear program \eqref{eq:fracclique}.

Recall from the introduction that $H_n\coloneqq \sum_{j=1}^n\frac{1}{j}$ for $n\geq1$. For convenience, define $H_0\coloneqq 0$. Our construction of a fractional clique is based on the \emph{harmonic tail function} $\ell_n:[n]\to \mathbb{R}$ defined by $\ell_n(i)=H_n-H_{i-1}$ or, in other words, $\ell_n(i)=\sum_{j=i}^n\frac{1}{j}$. Given a set $S\subseteq [n]$, define $\ell_n(S)\coloneqq \sum_{i\in S}\ell_n(i)$. We observe that
\begin{equation}
\label{eq: sum ell}
\ell_n([n])=\sum_{i=1}^n\ell_n(i)=\sum_{i=1}^n\sum_{j=i}^n\frac{1}{j}=n\end{equation}
as, for each $1\leq j\leq n$, the term $\frac{1}{j}$ appears $j$ times in the summation. We now state two lemmas and derive Theorem~\ref{theo: rg} from them; we will then devote the rest of the section to proving the lemmas.

\begin{lemma}
\label{lem: fractional clique}
Let $0< \varepsilon\le\frac14$. If $n\geq d\geq 2$ and
\begin{equation}
\label{eq: fractional clique}
\frac{d^3H_n^2}{n} \le2\varepsilon,
\end{equation}
then the probability that $U_{n,d}$ has an independent set $I$ such that $\ell_n(I)\geq  \left(1+\varepsilon\right)\frac{n(W(d)+1)}{d}$ is at most $\varepsilon$.
\end{lemma}

\begin{lemma}
\label{lem: count cycles}
For any $r\geq3$ and positive integers $d$ and $n$, the expected number of cycles of length $r$ in $U_{n,d}$ is at most
\[\frac{1}{2r}\binom{2r-2}{r-1}d^rH_n.\]
\end{lemma}

\begin{proof}[Proof of Theorem~\ref{theo: rg}]
Given $0<\delta<1$, set $\varepsilon:=\frac\delta4$. Consider the graph $U_{n,d}$ from Definition~\ref{def: uniform attachment}.  First, let us show that $n,d$ and $\varepsilon$ satisfy the hypothesis \eqref{eq: fractional clique} of Lemma~\ref{lem: fractional clique}. Since $g\geq4$, we have $\frac{1}{g-1}\binom{2g-4}{g-2}\geq2$ and $d^{g-1}\geq d^3$. Therefore, \eqref{eq:rg} implies
\[\frac{d^3H_n^2}{n}\leq \frac{\frac{1}{g-1}\binom{2g-4}{g-2}d^{g-1}H_n^2}{2n}\leq \frac{\delta}{2}=2\varepsilon.\]
Thus, by Lemma~\ref{lem: fractional clique}, with probability at least $1-\varepsilon$, every independent set $I$ in $U_{n,d}$ satisfies $\ell_n(I)\leq  \left(1+\varepsilon\right)\frac{n(W(d)+1)}{d}$. 

Next, define $f:[n]\to[0,\infty)$ by 
\[f(i)\coloneqq \frac{\ell_n(i)d}{(1+\varepsilon)n(W(d)+1)}.\]
Using the result from the previous paragraph, we get that, with probability at least $1-\varepsilon$, every independent set $I$ in $U_{n,d}$ satisfies 
\[\sum_{i\in I}f(i)=\sum_{i\in I}\frac{\ell_n(i)d}{(1+\varepsilon)n(W(d)+1)}=\frac{\ell_n(I)d}{(1+\varepsilon)n(W(d)+1)}\leq 1.\]
Thus, with probability at least $1-\varepsilon$, $f$ is a fractional clique in $U_{n,d}$. 

Let $R\subseteq [n]$ be obtained by selecting the maximal vertex under the natural ordering of $[n]$ from every cycle of length less than $g$ in $U_{n,d}$ and define $G=U_{n,d}-R$. Clearly, $G$ is $d$-degenerate and has girth at least $g$ by construction. If $f$ is a fractional clique in $U_{n,d}$, then its restriction to $V(G)$ is also a fractional clique in $G$. By Lemma~\ref{lem: count cycles} and the fact that $\ell_n(i)\leq \ell_n(1)=H_n$ for all $i\in [n]$, we get
\[\mathbb{E}[\ell_n(R)]\leq H_n^2\sum_{r=3}^{g-1}\frac{1}{2r}\binom{2r-2}{r-1}d^r.\]
We estimate the summation in the above expression in terms of its final summand. For any $r\geq4$,
\[\frac{\frac{1}{2(r-1)}\binom{2r-4}{r-2}d^{r-1}}{\frac{1}{2r}\binom{2r-2}{r-1}d^r} = \frac{r}{2d(2r-3)}\leq \frac{2}{5d}\leq \frac{1}{5}\]
where the last inequality uses that $d\geq2$. It follows that 
\[\sum_{r=3}^{g-1}\frac{1}{2r}\binom{2r-2}{r-1}d^r\leq \sum_{r=3}^{g-1}\left(\frac{1}{5}\right)^{g-1-r}\frac{1}{2(g-1)}\binom{2g-4}{g-2}d^{g-1}\leq \frac{5}{8(g-1)}\binom{2g-4}{g-2}d^{g-1}.\]
Therefore,
\[\mathbb{E}[\ell_n(R)]\leq \frac{5}{8(g-1)}\binom{2g-4}{g-2}d^{g-1}H_n^2\stackrel{\text{by\eqref{eq:rg}}}{\leq }\frac{5\delta}{8}n.\]
By Markov's Inequality and the fact that $\varepsilon=\frac\delta4$, we have
\[\mathbb{P}\left[\ell_n(R) > (\delta-(1-\delta)\varepsilon)n\right]\leq\frac{5\delta n/8}{(\delta-(1-\delta)\varepsilon)n}=\frac{5}{6+8\varepsilon} < 1-\varepsilon.\]
Therefore, 
\[\mathbb{P}\left[\ell_n(R) \leq (\delta-(1-\delta)\varepsilon)n\right]>\varepsilon.\]

So the restriction of $f$ to $V(G)$ is a fractional clique in $G$ with probability at least $1-\varepsilon$ and $\ell_n(R) \leq (\delta-(1-\delta)\varepsilon)n$ with probability greater than $\varepsilon$. Therefore, with positive probability, both events hold and we have, by weak duality,
\[
\chi_f(G)\geq \sum_{i\in V(G)}f(i) =\frac{\left(\ell_n([n])-\ell_n(R)\right)d}{(1+\varepsilon)n(W(d)+1)}\geq \frac{(1-\delta + (1-\delta)\varepsilon)nd}{(1+\varepsilon)n(W(d)+1)}=(1-\delta)\frac{d}{W(d)+1},
\]
completing the proof.
\end{proof}

\begin{remark}
Defining $G$ by deleting one vertex from every short cycle in $U_{n,d}$ was convenient for the analysis, as it ensures that the restriction of $f$ to $V(G)$ is still a fractional clique. It may be possible to obtain a sharper result by deleting one edge, rather than one vertex, from every such cycle; however, this would require proving a version of Lemma~\ref{lem: fractional clique} for all of the new independent sets $I$ which are created by deleting such edges. We have not explored this alternative strategy.
\end{remark}

We now turn our attention to the proof of Lemma~\ref{lem: fractional clique}, which applies the next two propositions.

\begin{proposition}
\label{prop:deterministic}
For any set $S\subseteq [n]$, 
\[\sum_{j\in S}\frac{|S\cap[j]|}{j}\geq \frac{\ell_n(S)^2 + |S|^2}{2n}.\]
\end{proposition}

\begin{proof}
We prove the equivalent inequality
\[2n\left(\sum_{j\in S}\frac{|S\cap[j]|}{j}\right)\geq \ell_n(S)^2 + |S|^2\]
by induction on $n$. In the base case $n=1$, if $S=\emptyset$, then both sides of the inequality are equal to $0$ and, if $S=\{1\}$, then both sides are equal to $2$. Thus, the inequality holds for $n=1$.

Let $n\geq2$ and define $S_0\coloneqq S\cap[n-1]$. By the induction hypothesis, we have
\begin{equation}
\label{eq:all but last}
2(n-1)\left(\sum_{j\in S_0}\frac{|S_0\cap[j]|}{j}\right)\geq \ell_{n-1}(S_0)^2 + |S_0|^2
\end{equation}
There are two cases to consider: namely, $n\notin S$ or $n\in S$. 

\begin{case}
$n\notin S$.
\end{case}

In this case, 
\begin{align*}
2n\left(\sum_{j\in S}\frac{|S\cap[j]|}{j}\right)&=\frac{n}{n-1}\left(2(n-1)\left(\sum_{j\in S_0}\frac{|S_0\cap[j]|}{j}\right)\right)\\
&\stackrel{\text{by \eqref{eq:all but last}}}{\geq} \frac{n}{n-1}\left(\ell_{n-1}(S_0)^2 + |S_0|^2\right).
\end{align*}
By definition, we have that $\ell_n(i)=\ell_{n-1}(i)+\frac{1}{n}$ for every $i\in [n-1]$. Therefore, $\ell_{n}(S)=\ell_{n-1}(S_0)+\frac{|S_0|}{n}$. So letting $\ell\coloneqq \ell_{n-1}(S_0)$ and $s\coloneqq |S_0|$, we see that, to complete the proof in this case, it suffices to show 
\begin{equation}\label{eq:n notin S}\frac{n}{n-1}\left(\ell^2 + s^2\right) \geq \left(\ell+\frac{s}{n}\right)^2+s^2.\end{equation}
The difference between the left and right sides of \eqref{eq:n notin S} is
\begin{align*}
\frac{n}{n-1}\left(\ell^2 + s^2\right) - \left(\ell+\frac{s}{n}\right)^2-s^2&=\frac{n}{n-1}\ell^2 + \frac{n}{n-1}s^2 - \ell^2-\frac{2s\ell}{n}-\frac{s^2}{n^2}-s^2\\
&=\frac{n^3\ell^2 + n^3s^2-n^2(n-1)\ell^2-2n(n-1)s\ell - s^2(n-1)-s^2n^2(n-1)}{n^2(n-1)}\\
&=\frac{(n\ell-(n-1)s)^2+ns^2}{n^2(n-1)},
\end{align*}
which is a sum of squares, and hence non-negative. Therefore, \eqref{eq:n notin S} holds. 

\begin{case}
$n\in S$.
\end{case}

In this case, we have
\begin{align*}
2n\left(\sum_{j\in S}\frac{|S\cap[j]|}{j}\right)&=2n\left(\sum_{j\in S_0}\frac{|S_0\cap[j]|}{j} + \frac{|S|}{n}\right)\\
&=\frac{n}{n-1}\left(2(n-1)\left(\sum_{j\in S_0}\frac{|S_0\cap[j]|}{j}\right)\right)+2|S|\\
&\stackrel{\text{by \eqref{eq:all but last}}}{\geq} \frac{n}{n-1}\left(\ell_{n-1}(S_0)^2 + |S_0|^2\right)+2(|S_0|+1).
\end{align*}
Letting $\ell\coloneqq \ell_{n-1}(S_0)$ and $s\coloneqq |S_0|$, it suffices to show that
\begin{equation}\label{eq:n in S}\frac{n}{n-1}\left(\ell^2 + s^2\right)+2(s+1) \geq \left(\ell+\frac{s+1}{n}\right)^2+(s+1)^2.\end{equation}
Similar to the previous case, a straightforward algebraic expansion shows that the difference between the left and right hand sides of \eqref{eq:n in S} is
\begin{align*}
\frac{(n\ell-(n-1)(s+1))^2+n(n-s-1)^2}{n^2(n-1)},
\end{align*}
which is a sum of squares, and hence non-negative. Therefore, \eqref{eq:n in S} holds. 
\end{proof}

Next, we use the previous proposition to get an upper bound on the probability that a fixed subset of $[n]$ is independent. 

\begin{proposition}
\label{prop:probability bound}
For any fixed set $S\subseteq [n]$,
\[\mathbb{P}(\text{$S$ is independent in $U_{n,d}$})\leq \exp\left(dH_n - \frac{d}{2n}\left(\ell_n(S)^2+|S|^2\right)\right).\]
\end{proposition}

\begin{proof}
A set $S\subseteq [n]$ is independent if and only if, for every vertex $j\in S$, the set $N_j^-$ is disjoint from $S$. If $j>d$, then there are precisely $\binom{j}{d}$ choices for the set $N_j^-\subseteq \{0,\dots,j-1\}$ and $\binom{j-|S\cap [j-1]|}{d}$ of them avoid $S$. Thus,
\begin{align*}
\mathbb{P}\left(N_j^-\cap S = \emptyset\right)&=\frac{\binom{j-|S\cap[j-1]|}{d}}{\binom{j}{d}}\\
&= \frac{(j-|S\cap[j-1]|)(j-|S\cap[j-1]|-1)\cdots(j - d -|S\cap[j-1]| + 1)}{j(j-1)\cdots(j - d + 1)}\\
&\leq \left(1-\frac{|S\cap[j-1]|}{j}\right)^{d}.
\end{align*}
Using the standard inequality $1-x\leq e^{-x}$ for $x\ge0$, we get
\begin{equation}
\label{eq:looking back}
\mathbb{P}\left(N_j^-\cap S = \emptyset\right)\leq \exp\left(-\frac{d|S\cap[j-1]|}{j}\right)
\end{equation}
for all $j>d$. We observe that the same inequality also holds for $j\leq d$. Indeed, every such vertex satisfies $N_j^-=\{0,\dots,j-1\}$. So if $S\cap [j-1]\neq \emptyset$, then the left side is zero and the inequality holds trivially; on the other hand, if $S\cap [j-1]= \emptyset$, then both sides of \eqref{eq:looking back} are $1$. So \eqref{eq:looking back} holds for all $1\leq j\leq n$. 

By \eqref{eq:looking back} and the fact that $N_i^-$ and $N_j^-$ for $i\neq j$ are chosen independently, we have
\[\mathbb{P}(\text{$S$ is independent in $U_{n,d}$})\leq \exp\left(-d\sum_{j\in S}\frac{|S\cap[j-1]|}{j}\right).\]
Finally, we have 
\[\sum_{j\in S}\frac{|S\cap[j-1]|}{j} \ge \sum_{j\in S}\frac{|S\cap[j]|-1}{j}= \sum_{j\in S}\frac{|S\cap[j]|}{j}-\sum_{j\in S}\frac{1}{j}\geq \frac{\ell_n(S)^2+|S|^2}{2n} - H_n\]
where the last inequality uses Proposition~\ref{prop:deterministic}. This completes the proof. 
\end{proof}

We are now ready to prove Lemma~\ref{lem: fractional clique}.

\begin{proof}[Proof of Lemma~\ref{lem: fractional clique}]
For each $0\leq s\leq n$, define
\[X_s\coloneqq \left|\left\{S\subseteq [n]: \text{$|S|=s$, $S$ is independent in $U_{n,d}$ and $\ell_n(S)\geq \left(1+\varepsilon\right)\frac{n(W(d)+1)}{d}$}\right\}\right|\]
and define $X\coloneqq \sum_{s=0}^nX_s$. By Markov's inequality, it suffices to show that, if $d$ and $n$ satisfy \eqref{eq: fractional clique}, then $\mathbb{E}[X]\leq \varepsilon$. 

By Proposition~\ref{prop:probability bound}, we have, for any $0\leq s\leq n$, a set $S\subseteq [n]$ of cardinality $s$ satisfying $\ell_n(S)\geq \left(1+\varepsilon\right)\frac{n(W(d)+1)}{d}$ is independent in $U_{n,d}$ with probability at most
\begin{align*}
    &\exp\left(dH_n - \frac{d}{2n}\left(\left(1+\varepsilon\right)^2\left(\frac{n(W(d)+1)}{d}\right)^2+s^2\right)\right) \\
    &\qquad \qquad \qquad \qquad = \exp\left(dH_n - \frac{nd}{2}\left(\left(1+\varepsilon\right)^2\left(\frac{W(d)+1}{d}\right)^2+\left(s/n\right)^2\right)\right).
\end{align*}
Therefore,
\[\mathbb{E}(X_s)\leq \binom{n}{s}\exp\left(dH_n - \frac{nd}{2}\left(\left(1+\varepsilon\right)^2\left(\frac{W(d)+1}{d}\right)^2+\left(s/n\right)^2\right)\right)\]
Recall that $\binom{n}{s}\leq \exp(nh(s/n))$ where $h(x)=x\log(1/x)+(1-x)\log(1/(1-x))$ for $x\in(0,1)$ and $h(0)=h(1)=0$ (this formula is equivalent to the usual one in base $2$ as opposed to base $e$). Therefore, 
\begin{align*}
\mathbb{E}(X_s)&\leq \exp\left(dH_n + n\left(h(s/n) -\frac{d}{2}\left(s/n\right)^2 - \frac{d}{2}\left(1+\varepsilon\right)^2\left(\frac{W(d)+1}{d}\right)^2\right)\right)\\
&\leq \exp\left(dH_n + n\left(\sup_{x\in [0,1]}\left[h(x) -\frac{d}{2}x^2\right] - \frac{\left(1+\varepsilon\right)^2(W(d)+1)^2}{2d}\right)\right).
\end{align*}
Since the function $\log(t)$ for $t>0$ is concave, it lies below its tangent line at $t=1$. Therefore, $\log(t)\leq t-1$ and so
\[h(x)=x\log(1/x)+(1-x)\log(1/(1-x))\leq x\log(1/x)+(1-x)(x/(1-x)) = x\log(1/x)+x.\]
So we have $h(x)-\frac{d}{2}x^2\leq \phi(x)$ where $\phi(x)\coloneqq x\log(1/x)+x - \frac{d}{2}x^2$ for $x\in(0,1]$ and $\phi(0)=0$. Then, for $x\in (0,1]$, we have $\phi'(x)=-\log(x)-dx$ and $\phi''(x)=-\frac{1}{x}-d<0$. Therefore, the unique maximum of $\phi(x)$ for $x\in [0,1]$ is achieved when $dx=\log(1/x)$ or, equivalently, $x=\frac{W(d)}{d}$. Thus, plugging this value in for $x$ yields 
\begin{align*}
\sup_{x\in[0,1]}\phi(x)&= \frac{W(d)}{d}\log(d/W(d)) + \frac{W(d)}{d} - \frac{W(d)^2}{2d}\\
&=\frac{W(d)^2+2W(d)}{2d}\\
&=\frac{(W(d)+1)^2-1}{2d}.
\end{align*}
Consequently, 
\begin{align*}
\mathbb{E}(X)&=\sum_{s=0}^n\mathbb{E}(X_s)\\
&\leq (n+1)\exp\left(dH_n + \frac{n}{2d}\left((W(d)+1)^2 - 1 - (1+\varepsilon)^2(W(d)+1)^2\right)\right)\\
&=(n+1)\exp\left(dH_n - \frac{n}{2d}\left(1 + (2\varepsilon+\varepsilon^2)(W(d)+1)^2\right)\right)\\
&\leq(n+1)\exp\left(dH_n - \frac{n}{2d}\right).
\end{align*}

It remains to show that $(n+1)\exp\left(dH_n - \frac{n}{2d}\right)\leq \varepsilon$. By \eqref{eq: fractional clique}, we know that $\frac{d^3H_n^2}{2n}\leq \varepsilon$ and so it suffices to prove that $(n+1)\exp\left(dH_n - \frac{n}{2d}\right)\leq \frac{d^3H_n^2}{2n}$ or, equivalently, that
\begin{equation}
\label{eq:finishingOff}
\log\left(\frac{(n+1)\exp\left(dH_n-\frac{n}{2d}\right)}{\frac{d^3H_n^2}{2n}}\right)<0.
\end{equation}
Since $\varepsilon\leq \frac14$, \eqref{eq: fractional clique} implies that $\frac{n}{2d}\geq d^2H_n^2$. Using this together with $H_n\geq \log(n+1)\geq \log n$, we get
\begin{align*}
\log\left(\frac{(n+1)\exp\left(dH_n-\frac{n}{2d}\right)}{\frac{d^3H_n^2}{2n}}\right) &= \log(n+1)+dH_n - \frac{n}{2d} + \log(2n)-3\log d-2\log H_n\\
&\leq(d+2)H_n - d^2H_n^2  +\log 2-3\log d-2\log H_n
\end{align*}
Since $d\geq2$ and $H_n\geq1$, we have $d^2H_n^2\geq (d+2)H_n$. Therefore, the sum of the first two terms in the above expression is non-positive. Also, since $d\geq2$, the sum of the last three terms is negative. Thus, \eqref{eq:finishingOff} holds, which completes the proof.
\end{proof}

Finally, we present the proof of Lemma~\ref{lem: count cycles}. 

\begin{proof}[Proof of Lemma~\ref{lem: count cycles}]
For any $1\leq i<j\leq n$ with $j\geq d$, the probability that the set $N_j^-\subseteq\{0,\dots,j-1\}$ contains $i$ is precisely  $\frac{\binom{j-1}{d-1}}{\binom{j}{d}}=\frac{d}{j}$. Moreover, the bound 
\begin{equation}
\label{eq:onePrev}
\mathbb{P}(i\in N_j^-)\leq \frac{d}{j}
\end{equation}
also holds true for $1\leq j<d$ as the right side would be greater than $1$. Similarly, for any $1\leq i_1<i_2<j\leq n$, we have 
\begin{equation}
\label{eq:twoPrev}\mathbb{P}(i_1,i_2\in N_j^-)\leq \frac{d(d-1)}{j(j-1)}\leq\left(\frac{d}{j}\right)^2.
\end{equation}

Let $(v_1,\dots,v_r)$ be an $r$-tuple of distinct vertices in $[n]$ and, for convenience, let $v_{r+1}\coloneqq v_1$. Our goal is to bound the probability that these $r$ vertices form a cycle in $U_{n,d}$ in this prescribed order. The presence or absence of an edge $v_iv_{i+1}$ in this cycle depends only on the choice of $N_{v_i}^-$ or $N_{v_{i+1}}^-$, depending on which of $v_i$ or $v_{i+1}$ appears later in $[n]$. Also, by construction, the sets $N_{v_i}^-$ and $N_{v_j}^-$ for $v_i\neq v_j$ are chosen independently of one another. Thus, by \eqref{eq:onePrev} and \eqref{eq:twoPrev}, we have 
\[\mathbb{P}(v_iv_{i+1}\in E(U_{n,d})\text{ for all }1\leq i\leq r)\leq d^r\prod_{i=1}^r\frac{1}{\max\{v_i,v_{i+1}\}}.\]
Therefore, the expected number of $r$-cycles in $U_{n,d}$ is at most
\[\frac{1}{2r}\sum_{\substack{(v_1,\dots,v_r)\in [n]^r\\ v_1,\dots,v_r\text{ distinct}}}d^r\prod_{i=1}^r\frac{1}{\max\{v_i,v_{i+1}\}}\]
where the factor of $1/2r$ comes from the fact that each cycle is counted $2r$ times. If we drop the condition that the vertices are distinct, then it only increases the sum. Thus, the expected number of $r$-cycles is at most 
\[\frac{d^r}{2r}\sum_{(v_1,\dots,v_r)\in [n]^r}\prod_{i=1}^r\frac{1}{\max\{v_i,v_{i+1}\}} = \frac{d^r}{2r}\Tr(A^r),\]
where $A$ is the $n\times n$ matrix with $A_{i,j}=\frac{1}{\max\{i,j\}}$ for all $1\leq i,j\leq n$. The proof of the lemma boils down to establishing the following:
\[\Tr(A^r)\leq \binom{2r-2}{r-1}H_n.\]
To this end, recall that 
\begin{equation}
\label{eq:Tr}
\Tr(A^r)=\sum_{(v_1,\dots,v_r)\in [n]^r}\prod_{i=1}^r\frac{1}{\max\{v_i,v_{i+1}\}}.
\end{equation}
Roughly speaking, we start by reducing the problem of computing $\Tr(A^r)$, which is analogous to counting closed walks of length   $r$ in a graph, to counting walks of length $r-2$ by conditioning on the choice of the largest index in $[n]$ that appears among $v_1,\dots,v_r$. To this end, for each $s,m\geq1$, define
\begin{equation}\label{eq:P_s(m)}P_s(m)\coloneqq \sum_{(x_1,\dots,x_s)\in [m]^s}\prod_{i=1}^{s-1}\frac{1}{\max\{x_i,x_{i+1}\}}.\end{equation}

Let us show, by induction on $s$, that
\begin{equation}
\label{eq:PCatalan}
P_s(m)\leq \frac{1}{s+1}\binom{2s}{s}m.
\end{equation}
In the case $s=1$, $P_1(m)$ is the sum over all $x_1\in [m]$ of an empty product, and so it is equal to $\sum_{x_1=1}^m1=m$. The right hand side of \eqref{eq:PCatalan} is also equal to $m$ in this case and so the result holds for $s=1$. Now, suppose that $s\geq2$. For a given choice of $(x_1,\dots,x_s)\in [m]^s$, we let $q$ be an index such that $x_q=\max\{x_1,\dots,x_s\}$. We bound $P_s(m)$ from above by summing the contribution to \eqref{eq:P_s(m)} of each choice of $q$; note that, for some tuples $(x_1,\dots,x_s)$, the choice of $q$ is not unique, and so this will be an overestimate. Additionally, when $q\in\{1,s\}$, the factor $\frac{1}{x_q}$ only appears once in $\prod_{i=1}^{s-1}\frac{1}{\max\{x_i,x_{i+1}\}}$ and, otherwise, it appears twice. So by induction, the contribution from $q=1$ is at most
\[\sum_{x_1=1}^m\frac{1}{x_1}P_{s-1}(x_1)\leq \sum_{x_1=1}^m\frac{1}{x_1}\cdot \frac{1}{s}\binom{2s-2}{s-1}x_1 = \frac{1}{s}\binom{2s-2}{s-1}\cdot m\]
and the same bound is valid for the contribution from $q=s$. Next, suppose that $2\leq q\leq s-1$. In this case, by induction, the contribution from $q$ is at most 
\begin{align*}\sum_{x_q=1}^m\frac{1}{x_q^2}P_{q-1}(x_q)P_{s-q}(x_q)&\leq \sum_{x_q=1}^m\frac{1}{x_q^2}\left(\frac{1}{q}\binom{2q-2}{q-1}x_q\right)\left(\frac{1}{s-q+1}\binom{2s-2q}{s-q}x_q\right) \\
&= \frac{1}{q}\binom{2q-2}{q-1}\cdot \frac{1}{s-q+1}\binom{2s-2q}{s-q}\cdot m\end{align*}
Putting this all together, we get
\[P_s(m)\leq m\left(\frac{2}{s}\binom{2s-2}{s-1} + \sum_{q=2}^{s-1}\left[\frac{1}{q}\binom{2q-2}{q-1}\cdot \frac{1}{s-q+1}\binom{2s-2q}{s-q}\right]\right)\]
By re-indexing the sum by $a=q-1$ and incorporating the other two terms, we get that
\[P_s(m)\leq m \sum_{a=0}^{s-1}\left[\frac{1}{a+1}\binom{2a}{a} \cdot\frac{1}{s-a}\binom{2s-2a-2}{s-a-1}\right] = \frac{1}{s+1}\binom{2s}{s}m\]
by the standard Catalan recurrence (see, e.g.,~\cite[(1.1)]{Stanley15}). Thus, \eqref{eq:PCatalan} is true. 

Now, to complete the proof of the claim, we bound $\Tr(A^r)$ in terms of the quantities $P_s(m)$ and apply \eqref{eq:PCatalan}. For a given choice of $(v_1,\dots,v_r)\in [n]^r$, let $q$ be an index such that $v_q=\max\{v_1,\dots,v_r\}$. Then, by symmetry, every value of $q$ contributes the same amount to \eqref{eq:Tr}, and so it suffices to consider the contribution of $q=r$, which is at most
\[\sum_{v_r=1}^n\frac{1}{v_r^2}P_{r-1}(v_r)\stackrel{\text{by \eqref{eq:PCatalan}}}{\leq} \frac{1}{r}\binom{2r-2}{r-1}\sum_{v_r=1}^n\frac{1}{v_r} = \frac{1}{r}\binom{2r-2}{r-1}H_n.\]
Summing over all $q=1,\dots,r$ contributes at most a factor of $r$, thereby proving \eqref{eq:Tr} and, hence, Lemma~\ref{lem: count cycles}. 
\end{proof}

\begin{remark}
The matrix $A$ in the proof of Lemma~\ref{lem: count cycles} is the same as the matrix with entries $k_{1/2}(i,j)$ where $k_{1/2}$ is defined as in~\cite[Example~3.2]{Pushnitski23}. So applying~\cite[Theorem~2.1]{Pushnitski23} with $g:\mathbb{R}\to \mathbb{R}$ being a Lipschitz continuous function such that $g(x)=x^r$ for all $x$ in an interval that contains the spectrum of $A$, we get that 
\[\lim_{n\to\infty}\frac{\Tr(A^r)}{\log n}= \binom{2r-2}{r-1}\]
which shows that the bound in the proof of Lemma~\ref{lem: count cycles} gives the correct leading asymptotic for $\Tr(A^r)$.
\end{remark}

\section{Concluding remarks}\label{section: concluding remarks}
Our result should be compared to Kim's proof~\cite{Kim} that girth-$5$ graphs with maximum degree $\Delta$ have chromatic number $(1+o(1))\frac{\Delta}{\log\Delta}$. 
We note that the girth-$5$ constraint can be relaxed to $C_4$-freeness in his proof.
Kim's strategy, using the nibble method, is very different from ours, and seems hard to generalize to the degeneracy setting.

\medskip

The outstanding open problem to resolve is Conjecture~\ref{conj:MS}\ref{conj: ub} of Martinsson and Steiner, i.e., that every triangle-free $d$-degenerate graph has fractional chromatic number at most $(1+o(1))\frac{d}{\log d}$. Recall that they proved an upper bound $(4+o(1))\frac{d}{\log d}$, and our result Theorem~\ref{theo: rg} shows that $(1+o(1))\frac{d}{\log d}$ would be best possible.

The second author~\cite{Dhawan} gave a new proof (and an extension to various related problems) of the result of Martinsson and Steiner~\cite{MartStein}, using a variant of Algorithm~\ref{algorithm: fcp}. Observe that in this algorithm we have a choice of how to set $\alpha$. We briefly explain why $\alpha\ge \frac{\log d}{d}$ is a bad choice, and it follows from our analysis that we cannot hope to establish a better bound than $\chi_f(G)\le \frac{d}{\log d}$ using Algorithm~\ref{algorithm: fcp}.

First, observe that if $v\in V(G)$ has $d$ left neighbors which form an independent set, and these vertices all have weight $\eta\frac{\log d}{d}$ for some $\eta\ge1$ at the time Algorithm~\ref{algorithm: fcp} processes them, then there are two possible outcomes. One is that a left neighbor is selected into $I$, in which case $v$ is not. The other is that no left neighbor is selected into $I$, in which case we observe the product
\[\alpha\left(1-\eta\frac{\log d}{d}\right)^{-d}\approx\alpha e^{\eta\log d}=\alpha d^\eta\]
is larger than $1$ and so certainly larger than we can choose $\hat{p}$ to be: again, the algorithm does not select $v$ into $I$.

This observation immediately explains why we cannot set $\alpha\ge\tfrac{\log d}{d}$: if a vertex $v$ has $d$ left-neighbors each of which has no left-neighbors itself, then the algorithm will never select $v\in I$.
It is tempting to believe that a better analysis for triangle-free graphs might improve the $4$. Unfortunately, this is not the case, as we now explain.

Suppose, as usual, that the initial weight of each vertex is $\alpha$ and define $\gamma$ to be $\frac{\log d}{\alpha d}$ so that $\alpha=\gamma^{-1}\frac{\log d}{d}$. Suppose that a vertex $v$ has $d$ left-neighbors forming an independent set $B$, each of which has left-neighborhood a set $A$ of $\frac{d\gamma\log\gamma }{\log d}$ vertices, with the vertices in $A$ having no left-neighbors (note that $A\cup B$ induces a complete bipartite graph). Again, there are two possibilities for the running of Algorithm~\ref{algorithm: fcp}. With probability
\[(1-\alpha)^{|A|}=(1-\alpha)^{ d\gamma\log\gamma/\log d}\approx \gamma^{-1}\]
no vertex of $A$ is selected into $I$, all the vertices of $B$ have weight at least $\gamma\alpha$, and if this is at least $\tfrac{\log d}{d}$ then $v$ is not selected into $I$. The other possibility is, with probability approximately $1-\gamma^{-1}$, at least one vertex of $A$ is selected into $I$, the weight of all vertices in $B$ is zero, and $v$ is selected into $I$ with probability $\alpha$. Thus, the probability that $v$ is selected is approximately $\alpha(1-\gamma^{-1})=\gamma^{-1}(1-\gamma^{-1})\frac{\log d}{d}$, which is at most $\frac{\log d}{4d}$, with the unique optimum being $\gamma=2$. This demonstrates that for graphs which may contain $4$-cycles, the analysis of~\cite{Dhawan} is asymptotically optimal.\footnote{We note that the argument of \cite{Dhawan} is in the setting of $r$-uniform hypergraphs having girth at least $4$. One can extract the argument for triangle-free graphs by setting $r = 2$.}

\medskip

The place where we use $C_4$-freeness in our argument is when we prove $(S_i)_{i=0}^{k-1}$ forms a supermartingale, specifically the last two cases when we process a vertex $v_i$ at distance at most $2$ from $v_k$, to argue that $|N_R(v_i) \cap N_L(v_k)| \leq 1$.
This condition is implied by forbidding triangles as well as the so-called ``Co-Neighbor Path'' configuration of $C_4$ in a degeneracy ordering of $G$ (see Fig.~\ref{fig:cycles}).
In particular, when considering triangle-free graphs, we may allow the graph to contain $C_4$'s of the other two possible configurations in a degeneracy ordering.
As a result, the Co-Neighbor Path configuration seems to be the main obstacle to proving Conjecture~\ref{conj:MS}\ref{conj: ub}.
Although, as discussed earlier, our approach cannot surpass the current bound of $4 + o(1)$ due to Martinsson and Steiner.

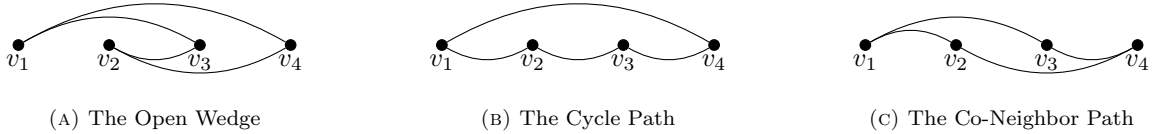
\begin{figure}[htb!]
 \begin{subfigure}{0.3\textwidth}
 \centering
     \begin{tikzpicture}[scale=0.6]
            \clip (-1,-1) rectangle (7,1);
            \node[circle,fill=black,draw,inner sep=0pt,minimum size=4pt] (a) at (0,0) {};
            \node[circle,fill=black,draw,inner sep=0pt,minimum size=4pt] (b) at (2,0) {};
            \node[circle,fill=black,draw,inner sep=0pt,minimum size=4pt] (c) at (4,0) {};
            \node[circle,fill=black,draw,inner sep=0pt,minimum size=4pt] (d) at (6,0) {};

            \node[anchor=north] (v1) at (a) {$v_1$};
            \node[anchor=north] (v2) at (b) {$v_2$};
            \node[anchor=north] (v3) at (c) {$v_3$};
            \node[anchor=north] (v4) at (d) {$v_4$};

            \draw (a) edge[bend left] (d) (a) edge[bend left] (c) (b) edge[bend right] (d) (b) edge[bend right] (c);
        \end{tikzpicture}
        \caption{The Open Wedge}
 \end{subfigure}
 \hfill
 \begin{subfigure}{0.3\textwidth}
 \centering
     \begin{tikzpicture}[scale=0.6]
            \clip (-1,-1) rectangle (7,1);
            \node[circle,fill=black,draw,inner sep=0pt,minimum size=4pt] (a) at (0,0) {};
            \node[circle,fill=black,draw,inner sep=0pt,minimum size=4pt] (b) at (2,0) {};
            \node[circle,fill=black,draw,inner sep=0pt,minimum size=4pt] (c) at (4,0) {};
            \node[circle,fill=black,draw,inner sep=0pt,minimum size=4pt] (d) at (6,0) {};

            \node[anchor=north] (v1) at (a) {$v_1$};
            \node[anchor=north] (v2) at (b) {$v_2$};
            \node[anchor=north] (v3) at (c) {$v_3$};
            \node[anchor=north] (v4) at (d) {$v_4$};

            \draw (a) edge[bend left] (d) (a) edge[bend right] (b) (b) edge[bend right] (c) (c) edge[bend right] (d);
        \end{tikzpicture}
        \caption{The Cycle Path}
 \end{subfigure}
 \hfill
 \begin{subfigure}{0.3\textwidth}
 \centering
     \begin{tikzpicture}[scale=0.6]
            \clip (-1,-1) rectangle (7,1);
            \node[circle,fill=black,draw,inner sep=0pt,minimum size=4pt] (a) at (0,0) {};
            \node[circle,fill=black,draw,inner sep=0pt,minimum size=4pt] (b) at (2,0) {};
            \node[circle,fill=black,draw,inner sep=0pt,minimum size=4pt] (c) at (4,0) {};
            \node[circle,fill=black,draw,inner sep=0pt,minimum size=4pt] (d) at (6,0) {};

            \node[anchor=north] (v1) at (a) {$v_1$};
            \node[anchor=north] (v2) at (b) {$v_2$};
            \node[anchor=north] (v3) at (c) {$v_3$};
            \node[anchor=north] (v4) at (d) {$v_4$};

            \draw (a) edge[bend left] (c) (a) edge[bend left] (b) (b) edge[bend right] (d) (c) edge[bend right] (d);
        \end{tikzpicture}
        \caption{The Co-Neighbor Path}
 \end{subfigure}
     \caption{Possible configurations of $C_4$ in $G$.}
    \label{fig:cycles}

\end{figure}

We note that we may relax the $C_4$-freeness condition further.
Indeed, it suffices to insist that for every pair of vertices $v_i$ and $v_k$, the set $N_R(v_i) \cap N_L(v_k)$ is sufficiently small (but possibly $\gg 1$).
More formally, if this set contained the $\ell$ vertices $u_1,\dots,u_\ell$  we would need to know
\[\E\left[\prod_{j=1}^\ell(1+p_i(u_j)\big|p_{i-1}(\cdot)\right]\] is sufficiently close to
\[\prod_{j=1}^\ell \left(1+\E\left[p_i(u_j)\big|p_{i-1}(\cdot)\right]\right)\,,\]
which, if $\ell$ is much smaller than $\hat{p}^{-1}$, is easily seen to be true. This bound $\ell\ll\hat{p}^{-1}$ is presumably not sharp (any given vertex is not too likely to get weight close to $\hat{p}$) and it would be interesting to determine whether a codegree bound of the form $d^\gamma$ for some positive $\gamma$ can be achieved.
We note that a proof for $\gamma = 1 - o(1)$ would strictly generalize a result of Campos, Jenssen, Michelen, and Sahasrabudhe~\cite{campos2023new} on the independence number of graphs with maximum codegree at most $\Delta^{1-o(1)}$; see, also, recent work of Bradshaw, Methuku, Wigal, and the second author~\cite{bradshaw2025toward} for an extension of their result to the chromatic number.

\begin{problem}
    Let $G$ be a $d$-degenerate graph satisfying $|N_L(u) \cap N_R(v)| \leq d^{1-o(1)}$ in some degeneracy ordering.
    Is $\chi_f(G) \leq (1 + o(1))\frac{d}{\log d}$?
\end{problem}

In a related direction, Anderson, Bernshteyn, and the second author~\cite{AndersonBernshteynDhawan} extended Kim's bound to $K_{t,t}$-free graphs $G$, showing that $\chi(G) \le (1 + o(1))\frac{\Delta}{\log \Delta}$. 
A key structural observation in their proof is that for every vertex $v \in V(G)$, there are relatively few vertices $u$ sharing a large common neighborhood with $v$. 
Leveraging this property within a nibble framework yields the desired bound. 
It is then natural to ask whether an analogous result holds for fractional coloring in the degeneracy setting.

\begin{problem}
    Let $G$ be a $d$-degenerate $K_{t, t}$-free graph.
    Is $\chi_f(G) \leq (1 + o(1))\frac{d}{\log d}$?
\end{problem}

\medskip

To understand where the factor-$2$ gap between the lower bounds for maximum degree $\Delta$ and degeneracy $d$ comes from, observe that the lower bound construction for maximum degree $\Delta$ has average degree close to $\Delta$, whereas our $d$-degenerate lower bound construction has average degree close to $2d$: this $2$ is responsible for the difference.

It is interesting to ask for an intuition as to why there is (conjecturally) a computational complexity barrier in the bounded degree setting but not the bounded degeneracy one.
For our discussion regarding this intuition, consider the algorithm which applies Algorithm~\ref{algorithm: fcp} to construct $q$ independent sets concurrently, i.e., at iteration $i$, the algorithm determines whether $v_i \in I_c$ for all $c \in [q]$ before proceeding further. Clearly, this procedure constructs an $(\alpha, q)$-coloring of $G$.
Suppose we color vertex by vertex (as our algorithm does) in degeneracy order (or, in the bounded degree setting, in any order). We now explain why we do not expect any such algorithm (in particular, ours) to succeed in the bounded degree setting.

At each step, the expected number of ways to continue the coloring should be at least $1$, otherwise we expect to fail rapidly. This expected number will have to do with the number of constraints on our coloring, which is given by the $d$ left-neighbors, independent of where in the degeneracy order we are; our choice of $\alpha$ in $(\alpha,q)$-coloring ensures that the expected number of choices is always a bit more than $1$. Since we get roughly the same expected number every time, this corresponds fairly well to the threshold $\alpha$ which makes the first moment of the total number of $(\alpha,q)$-colorings tend to infinity with $n$ as opposed to zero.

If we try the same heuristic for the bounded degree setting, with degrees bounded by $2d$ (so that the average degree is the same as for the $d$-degenerate case), for the same choice of $\alpha$ we get roughly the same first moment of the total number of $(\alpha,q)$-colorings. But when we look at coloring vertex by vertex, at the beginning of the order we generally expect to see very few constraints (so the expected number of ways to color is huge, but we can only use one), in the middle we see about $d$ constraints (and as above, the expected number of ways to continue is about $1$), and after the middle the number of constraints grows to $2d$, and we typically will fail to continue our coloring very quickly after the middle. For an approach like this to work, we would need to keep track of all the valid colorings early on and expect that a tiny but non-zero fraction of the valid colorings we build over the first half of the vertices turn out to extend to valid colorings of the whole graph. We do not in any case know how to analyze this rigorously, and from an algorithmic point of view, this means keeping track of exponentially many colorings which we cannot do efficiently.

For Algorithm~\ref{algorithm: fcp}, we can be a bit more specific about how it fails in the bounded degree setting. Having fixed an order, suppose $v$ is a vertex with significantly more than $d$ left-neighbors in the order. The sum of the weights of the left-neighbors starts significantly larger than $\alpha d$, and by the martingale property, in expectation this sum remains the same. But, as we explained above where pointing out why we cannot improve the factor $4$ without a girth assumption, this means we expect that either a neighbor of $v$ will be chosen into $I$, or $v$ will become bad: either way we do not expect to choose $v$ into $I$. In other words, we expect the critical bound $\Pr(v\in I)\ge(1-o(1))\alpha$ will fail for vertices $v$ with significantly more than $d$ left-neighbors, and consequently we will not assign $v$ sufficient colors to make an $(\alpha,q)$-coloring when we run Algorithm~\ref{algorithm: fcp} $q$ times.

\medskip

The independent sets we find in the fractional coloring in Theorem~\ref{theo: girth 5} are of expected size $(1+o(1))\frac{n\log d}{d}$ scattered fairly evenly through the degeneracy order (as opposed to more vertices appearing late in the order, for example). Similarly, these are the worst case independent sets for Theorem~\ref{theo: rg}. The proof does not characterize the worst case sets this accurately: but observe that in the proof of Lemma~\ref{lem: fractional clique}, the maximum is for $x=(1+o(1))\frac{\log d}{d}$, which corresponds to independent sets $S$ of size $(1+o(1))\tfrac{n\log d}{d}$ which satisfy $\ell_n(S)\approx\frac{n\log d}{d}$. Since $\ell_n([n])=n$, this is consistent with $S$ being uniformly scattered over $[n]$ (there are other ways to arrange for these two constraints to hold, which will make a smaller contribution when lower order terms are considered, but these are not required for the proof).

However, the random model used to prove Theorem~\ref{theo: rg} does contain some significantly larger independent sets. Proposition~\ref{prop:probability bound} actually gives a fairly good estimate of the probability that a given $S$ will be independent, and we see that if we choose $S$ larger, but such that $\ell_n(S)$ is smaller, we keep the probability of independence relatively high. This corresponds to sets $S$ which contain lots of elements later on in the degeneracy order and few early ones. We should stress that this argument is quite heuristic and even done accurately would only find the optimal $s$ such that the first moment of the number of independent sets of size $s$ tends to infinity. To prove that there is likely to exist an independent set of this size, one would need a further argument (such as a second moment calculation).

\vspace{2mm}
\subsection*{Acknowledgments}

This work was initiated during the workshop ``Cross-Community Collaborations in Combinatorics'' at the Banff International Research Station in June, 2026. We thank Natasha Morrison, Jozef Skokan, and Evelyne Smith-Roberge for organizing the workshop and inviting us, and all of the participants for stimulating discussions.
We also thank Seth Pettie for comments on an earlier version of this manuscript.

The authors used ChatGPT 5.5 Pro to aid in the proof of Theorem~\ref{theo: rg}. In particular, it was used to solve an optimization problem that the authors formulated to determine the correct shape of the fractional clique function, to check and simplify some probabilistic and algebraic estimates, and to help draft preliminary versions of some calculations. It also directed our attention to the reference~\cite{Pushnitski23}. The construction and overall strategy of the proof were formulated and the final exposition was written by the authors, who take full responsibility for the content.

\vspace{2mm}
\printbibliography

@book {Mitzenmacher,
    AUTHOR = {Mitzenmacher, Michael and Upfal, Eli},
     TITLE = {Probability and computing},
   EDITION = {Second},
      NOTE = {Randomization and probabilistic techniques in algorithms and
              data analysis},
 PUBLISHER = {Cambridge University Press, Cambridge},
      YEAR = {2017},
     PAGES = {xx+467},
      ISBN = {978-1-107-15488-9},
   MRCLASS = {68-01 (60C05 60G42 60J10 60K25 62H30 68W20 68W40)},
  MRNUMBER = {3674428},
}

@book {MolloyReed,
    AUTHOR = {Molloy, M. and Reed, B.},
     TITLE = {Graph Colouring and the Probabilistic Method},
    SERIES = {Algorithms and Combinatorics},
    VOLUME = {23},
 PUBLISHER = {Springer-Verlag, Berlin},
      YEAR = {2002},
     PAGES = {xiv+326},
      ISBN = {3-540-42139-4},
   MRCLASS = {05-02 (05C15 05C80 60-02 60C05)},
  MRNUMBER = {1869439},
MRREVIEWER = {P.\ Mark\ Kayll},
       DOI = {10.1007/978-3-642-04016-0},
       URL = {https://doi.org/10.1007/978-3-642-04016-0},
}

@article {FriezePerezGimenezPralatReiniger19,
    AUTHOR = {Frieze, A. and P\'erez-Gim\'enez, X. and Pra\l at,
              P. and Reiniger, B.},
     TITLE = {Perfect matchings and {H}amiltonian cycles in the preferential
              attachment model},
   JOURNAL = {Random Structures Algorithms},
  FJOURNAL = {Random Structures \& Algorithms},
    VOLUME = {54},
      YEAR = {2019},
    NUMBER = {2},
     PAGES = {258--288},
     DOI = {10.1002/rsa.20778},
URL = {https://doi.org/10.1002/rsa.20778}
}

@article {MalyshkinZhukovskii22,
    AUTHOR = {Malyshkin, Y. A. and Zhukovskii, M. E.},
     TITLE = {{$\gamma$}-variable first-order logic of uniform attachment
              random graphs},
   JOURNAL = {Discrete Math.},
  FJOURNAL = {Discrete Mathematics},
    VOLUME = {345},
      YEAR = {2022},
    NUMBER = {5},
     PAGES = {Paper No. 112802, 12},
     DOI = {10.1016/j.disc.2022.112802},
URL = {https://doi.org/10.1016/j.disc.2022.112802}
}

@article {AcanPittel20,
    AUTHOR = {Acan, H. and Pittel, B.},
     TITLE = {On connectivity, conductance and bootstrap percolation for a
              random {$k$}-out, age-biased graph},
   JOURNAL = {Random Structures Algorithms},
  FJOURNAL = {Random Structures \& Algorithms},
    VOLUME = {56},
      YEAR = {2020},
    NUMBER = {1},
     PAGES = {37--62},
     DOI = {10.1002/rsa.20872},
URL = {https://doi.org/10.1002/rsa.20872}
}

@book {Stanley15,
    AUTHOR = {Stanley, R. P.},
     TITLE = {Catalan numbers},
 PUBLISHER = {Cambridge University Press, New York},
      YEAR = {2015},
      DOI = {10.1017/CBO9781139871495},
URL = {https://doi.org/10.1017/CBO9781139871495}
}

@article {Pushnitski23,
    AUTHOR = {Pushnitski, A.},
     TITLE = {The spectral density of {H}ardy kernel matrices},
   JOURNAL = {J. Operator Theory},
  FJOURNAL = {Journal of Operator Theory},
    VOLUME = {89},
      YEAR = {2023},
    NUMBER = {1},
     PAGES = {3--21}
}

@misc{Dhawan,
      title={Fractional coloring via entropy}, 
      author={Abhishek Dhawan},
      year={2026},
      eprint={2603.17730},
      archivePrefix={arXiv},
      primaryClass={math.CO},
      url={https://arxiv.org/abs/2603.17730}, 
}

@misc{bradshaw2025toward,
  author = {Bradshaw, Peter and Dhawan, Abhishek and Methuku, Abhishek and Wigal, Michael C.},
  title = {Toward {V}u's conjecture},
  year = {2025},
  eprint={2508.16818},
      archivePrefix={arXiv},
      primaryClass={math.CO},
      url={https://arxiv.org/abs/2508.16818},
}

@article {MartStein,
    AUTHOR = {Martinsson, Anders and Steiner, Raphael},
     TITLE = {Random independent sets in triangle-free graphs},
   JOURNAL = {Forum Math. Sigma},
  FJOURNAL = {Forum of Mathematics. Sigma},
    VOLUME = {13},
      YEAR = {2025},
     PAGES = {Paper No. e156, 19},
     DOI = {10.1017/fms.2025.10112},
URL = {https://doi.org/10.1017/fms.2025.10112}
}

@unpublished{johansson,
	author = {A. Johansson},
	title = {Asymptotic choice number for triangle free graphs},
	howpublished = {Unpublished manuscript},
	date = {1996},
}

@article{shearer,
  title={A note on the independence number of triangle-free graphs},
  author={Shearer, James B},
  journal={Discrete Mathematics},
  volume={46},
  number={1},
  pages={83--87},
  year={1983},
  publisher={Elsevier},
  DOI = {10.1016/0012-365X(83)90273-X},
URL = {https://doi.org/10.1016/0012-365X(83)90273-X}
}

@article{Molloy,
    author = {M. Molloy},
    title = {The list chromatic number of graphs with small clique number},
    journaltitle = {J. Combin. Theory},
    series = {B},
    volume = {134},
    pages = {264--284},
    date = {2019},
    DOI = {10.1016/j.jctb.2018.06.007},
URL = {https://doi.org/10.1016/j.jctb.2018.06.007}
}

@article{AKS,
    AUTHOR = "N. Alon and M. Krivelevich and B. Sudakov",
    TITLE = "{Coloring graphs with sparse neighborhoods}",
    JOURNAL = "J. Combin. Theory",
    series = {B},
    YEAR = "1999",
    volume = {77},
    pages = {73--82},
    DOI = {10.1006/jctb.1999.1910},
URL = {https://doi.org/10.1006/jctb.1999.1910}
}

@article{bernshteyn2023counting,
  title={Counting colorings of triangle-free graphs},
  author={Bernshteyn, Anton and Brazelton, Tyler and Cao, Ruijia and Kang, Akum},
  journal={Journal of Combinatorial Theory, Series B},
  volume={161},
  pages={86--108},
  year={2023},
  publisher={Elsevier},
  DOI = {10.1016/j.jctb.2023.02.004},
URL = {https://doi.org/10.1016/j.jctb.2023.02.004}
}

@inproceedings{hurley2023uniformly,
  title={Uniformly random colourings of sparse graphs},
  author={Hurley, Eoin and Pirot, Fran{\c{c}}ois},
  booktitle={Proceedings of the 55th Annual ACM Symposium on Theory of Computing},
  pages={1357--1370},
  year={2023},
  DOI = {10.1145/3564246.3585242},
URL = {https://doi.org/10.1145/3564246.3585242}
}

@misc{BFSSZ,
  title={Coloring small locally sparse degenerate graphs and related problems},
  author={Brada{\v{c}}, Domagoj and Fox, Jacob and Steiner, Raphael and Sudakov, Benny and Zhang, Shengtong},
  year={2026},
  eprint={2601.15245},
      archivePrefix={arXiv},
      primaryClass={math.CO},
      url={https://arxiv.org/abs/2601.15245}, 
}

@article{EKT,
  title={Separation choosability and dense bipartite induced subgraphs},
  author={Esperet, Louis and Kang, Ross J and Thomass{\'e}, St{\'e}phan},
  journal={Combinatorics, Probability and Computing},
  volume={28},
  number={5},
  pages={720--732},
  year={2019},
  publisher={Cambridge University Press},
  DOI = {10.1017/S0963548319000026},
URL = {https://doi.org/10.1017/S0963548319000026}
}

@article{harris,
  title={Some results on chromatic number as a function of triangle count},
  author={Harris, David G},
  journal={SIAM Journal on Discrete Mathematics},
  volume={33},
  number={1},
  pages={546--563},
  year={2019},
  publisher={SIAM},
  DOI = {10.1137/17M115918X},
URL = {https://doi.org/10.1137/17M115918X}
}

@misc{HHKP,
      title={Improving ${R}(3,k)$ in just two bites}, 
      author={Zion Hefty and Paul Horn and Dylan King and Florian Pfender},
      year={2026},
      eprint={2510.19718},
      archivePrefix={arXiv},
      primaryClass={math.CO},
      url={https://arxiv.org/abs/2510.19718}, 
}

@article{AndersonBernshteynDhawan,
  author = {Anderson, James and Bernshteyn, Anton and Dhawan, Abhishek},
  title = {Coloring graphs with forbidden bipartite subgraphs},
  journal = {Combinatorics, Probability and Computing},
  volume = {32},
  number = {1},
  pages = {45--67},
  year = {2023},
  DOI = {10.1017/S0963548322000104},
URL = {https://doi.org/10.1017/S0963548322000104}
}

@misc{Kttt,
  author = {Dhawan, Abhishek and Janzer, Oliver and Methuku, Abhishek},
  title = {Independent sets and colorings of $K_{t,t,t}$-free graphs},
  year = {2025},
      eprint={2511.17191},
      archivePrefix={arXiv},
      primaryClass={math.CO},
      url={https://arxiv.org/abs/2511.17191}, 
}

@misc{campos2023new,
  author = {Campos, Marcelo and Jenssen, Matthew and Michelen, Marcus and Sahasrabudhe, Julian},
  title = {A new lower bound for sphere packing},
  year = {2023},
  eprint={2312.10026},
      archivePrefix={arXiv},
      primaryClass={math.MG},
      url={https://arxiv.org/abs/2312.10026}, 
}

@article{bernshteyn2019johansson,
  title={The Johansson-Molloy theorem for DP-coloring},
  author={Bernshteyn, Anton},
  journal={Random Structures \& Algorithms},
  volume={54},
  number={4},
  pages={653--664},
  year={2019},
  publisher={Wiley Online Library},
  DOI = {10.1002/rsa.20811},
URL = {https://doi.org/10.1002/rsa.20811}
}

@unpublished{DKPS,
  author = {Davies, Ewan and Kang, Ross J. and Pirot, Fran{\c{c}}ois and Sereni, Jean-S{\'e}bastien},
  title = {Graph structure via local occupancy},
  year = {2020},
  eprint={2003.14361},
      archivePrefix={arXiv},
      primaryClass={math.CO},
      url={https://arxiv.org/abs/2003.14361}, 
}

@article {Kim,
    AUTHOR = {Kim, Jeong Han},
     TITLE = {On {B}rooks' theorem for sparse graphs},
   JOURNAL = {Combin. Probab. Comput.},
  FJOURNAL = {Combinatorics, Probability and Computing},
    VOLUME = {4},
      YEAR = {1995},
    NUMBER = {2},
     PAGES = {97--132},
     DOI = {10.1017/S0963548300001528},
URL = {https://doi.org/10.1017/S0963548300001528}
}

@techreport{Karp,
    Author= {Karp, Richard M.},
    Title= {The Probabilistic Analysis of some Combinational Search Algorithms},
    Year= {1976},
    Month= {Apr},
    Url= {http://www2.eecs.berkeley.edu/Pubs/TechRpts/1976/28848.html},
    Number= {UCB/ERL M581},
}

@article {FGM,
    AUTHOR = {Fiz Pontiveros, G. and Griffiths, S. and Morris,
              R.},
     TITLE = {The triangle-free process and the {R}amsey number {$R(3,k)$}},
   JOURNAL = {Mem. Amer. Math. Soc.},
  FJOURNAL = {Memoirs of the American Mathematical Society},
    VOLUME = {263},
      YEAR = {2020},
    NUMBER = {1274},
     PAGES = {v+125},
       DOI = {10.1090/memo/1274},
       URL = {https://doi.org/10.1090/memo/1274},
}

@article {BohmanKeevash,
    AUTHOR = {Bohman, T. and Keevash, P.},
     TITLE = {Dynamic concentration of the triangle-free process},
   JOURNAL = {Random Structures Algorithms},
  FJOURNAL = {Random Structures \& Algorithms},
    VOLUME = {58},
      YEAR = {2021},
    NUMBER = {2},
     PAGES = {221--293},
       DOI = {10.1002/rsa.20973},
       URL = {https://doi.org/10.1002/rsa.20973},
}

@article {FL,
    AUTHOR = {Frieze, A. M. and \L uczak, T.},
     TITLE = {On the independence and chromatic numbers of random regular
              graphs},
   JOURNAL = {J. Combin. Theory Ser. B},
  FJOURNAL = {Journal of Combinatorial Theory. Series B},
    VOLUME = {54},
      YEAR = {1992},
    NUMBER = {1},
     PAGES = {123--132},
       DOI = {10.1016/0095-8956(92)90070-E},
       URL = {https://doi.org/10.1016/0095-8956(92)90070-E},
}

\end{document}